\documentclass[11pt,a4paper,reqno]{amsart}
\usepackage{amsaddr}
\usepackage{amsmath,amsfonts,amssymb,amsthm,mathrsfs}
\usepackage[utf8]{inputenc}
\usepackage[T1]{fontenc}
\usepackage{lmodern}
\usepackage{latexsym}
\usepackage{cancel}
\usepackage{microtype}
\usepackage{caption}
\usepackage{MnSymbol}
\usepackage{xcolor}
\usepackage{enumerate}
\usepackage[normalem]{ulem}
\usepackage{bbm}
\usepackage{geometry}
\usepackage{marginnote}

\usepackage{amsaddr}
\usepackage{lipsum}
\makeatletter
\g@addto@macro{\endabstract}{\@setabstract}
\newcommand{\authorfootnotes}{\renewcommand\thefootnote{\@fnsymbol\c@footnote}}%
\makeatother
\usepackage{microtype}

\usepackage[colorlinks=true,linkcolor=blue,citecolor=magenta]{hyperref}

\usepackage{graphics,tikz}

\newcommand{\BB}{{\mathcal B}}

\newcommand{\EE}{{\mathcal E}}
\newcommand{\FF}{{\mathcal F}}
\newcommand{\GG}{{\mathcal G}}
\newcommand{\HH}{{\mathcal H}}

\newcommand{\MM}{{\mathcal M}}

\newcommand{\BR}{{\mathbb R}}

\newcommand{\esssup}{\mathop{\mathrm{ess\,sup}}}

\newtheorem{theorem}{\bf Theorem}[section]
\newtheorem{proposition}[theorem]{\bf Proposition}
\newtheorem{lemma}[theorem]{\bf Lemma}
\newtheorem{corollary}[theorem]{\bf Corollary}

\theoremstyle{definition}
\newtheorem{definition}[theorem]{Definition}

\newtheorem{remark}[theorem]{Remark}

\numberwithin{equation}{section}

\renewcommand{\thefootnote}{{}}
\usepackage{setspace}  

\begin{document}

\title[Semilinear equations]{Nonlinear elliptic equations with integro-differential \\ divergence form operators and measure data \\under 
sign condition on the nonlinearity}


\maketitle
\begin{center}
 \normalsize
  \authorfootnotes
  TOMASZ KLIMSIAK\footnote{e-mail: {\tt tomas@mat.umk.pl}}\textsuperscript{1,2}   \par \bigskip

  \textsuperscript{1} {\small Institute of Mathematics, Polish Academy Of Sciences,\\
 \'{S}niadeckich 8,   00-656 Warsaw, Poland} \par \medskip
 
  \textsuperscript{2} {\small Faculty of
Mathematics and Computer Science, Nicolaus Copernicus University,\\
Chopina 12/18, 87-100 Toru\'n, Poland }\par
\end{center}

\begin{abstract}
We study existence problem for  semilinear equations with Borel measure data and  operator generated by a symmetric Markov semigroup.
We assume merely that the nonlinear part satisfies the so-called sign condition. Using the method of sub and supersolutions
we show the existence of maximal measure for which there exists a solution to the problem (the so-called reduced measure introduced by H. Brezis, M. Marcus and A.C. Ponce).
\end{abstract}

\footnotetext{{\em Mathematics Subject Classification:}
Primary ; Secondary .}

\footnotetext{{\em Keywords:}}

\section{Introduction}
{\bf Formulation of the problem}. Let $E$ be a locally compact separable metric space and $m$ be a full support positive Radon
measure on $E$. Consider a self-adjoint operator $(A,D(A))$ on $L^2(E;m)$ generating a Markov semigroup $(T_t)$,
a Caratheodory function $f:E\times \mathbb R\to \mathbb R$ satisfying the sign condition:
\begin{equation}
\label{eq2.1.1}
f(x,u)\cdot u\le 0,\quad u\in \mathbb R,\,\, m\mbox{-a.e.}\,\, x\in E,
\end{equation} 
and  Borel measure $\mu$ on $E$ which obeys $\int_E\rho\,d|\mu|<\infty$  for a strictly positive weight $\rho:E\to\mathbb R$
(class of such measures shall be denoted by $\MM_\rho$). In the present paper, we investigate the existence problem for
the following equation:
\begin{equation}
\label{eq1.1}
-Au=f(\cdot,u)+\mu.
\end{equation}
Throughout the paper, we assume that there exists the Green function $G$ for $A$ (see Section \ref{sec2.2}),
and $T_t\rho\le \rho\, m$-a.e. for any $t\ge 0$ (e.g. $\rho\equiv 1$ or $\rho$ being the principal eigenfunction for $-A$
satisfies the last requirement). 
The model example of purely non-local   operator, which fits into our framework  is
\begin{equation}
\label{eq1.3a0}
A u(x)=\lim_{r\searrow 0} \int_{\BR^d\setminus B(x,r)} \frac{a(x,y)(u(x)-u(y))}{|x-y|^{d}\varphi(|x-y|)}\,dy,
\end{equation}
with symmetric $a$ bounded between two strictly positive constants and $\varphi$ satisfying some standard growth assumptions (see Section \ref{sec3}).
In particular, taking $a\equiv const, \varphi(r)=r^{2\alpha}$ gives the fractional Laplacian (i.e. $A=-(-\Delta)^\alpha$), and 
taking $1/\varphi(r)=\int_0^1r^{-2\alpha}\,\nu(d\alpha)$, with $\nu$ being  a positive finite Borel measure compactly  supported  in $(0,1)$, gives 
a mixed relativistic symmetric stable operator  (see \cite{CK}).

A model example of a local operator fitting into our framework is 
\begin{equation}
\label{eq1.3}
Au(x)=\sum_{i,j=1}^d(a_{i,j}(x)u_{x_i})_{x_j},
\end{equation}
where  
$a=[a_{i,j}]_{i,j=1}^d$ is a  symmetric matrix  with Borel measurable locally integrable entries satisfying 
 ellipticity condition:
 \[
w(x)|\xi|^2\le\sum_{i,j=1}^da_{i,j}(x)\xi_i\xi_j\le v(x)|\xi|^2,\quad \xi,x\in\BR^d, 
 \]
where $v,w$ are suitable positive functions (see e.g. \cite{CW}).
For other examples of noteworthy classes of  operators, which the theory of the present paper is applicable to, see Section \ref{sec3}.

By a solution to \eqref{eq1.1}, we mean a Borel function $u$ on $E$ (finite $m$-a.e.) such that
$f(\cdot,u)\in L^1(E;\rho\cdot m)$, and
\begin{equation}
\label{eq1.2}
u(x)=\int_E G(x,y)f(y,u(y))\,m(dy)+\int_E G(x,y)\,\mu(dy),\quad m\mbox{-a.e.}\,\, x\in E.
\end{equation}
Let $(\EE,D(\EE))$ be a Dirichlet form generated by $(A,D(A))$:
\[
\EE(w,v):=(\sqrt{-A}w,\sqrt{-A}v),\quad  w,v\in D(\EE):=D(\sqrt {-A}),
\]
and $Cap_A$ be a capacity generated by $A$:  for open $U\subset E$, 
\begin{equation}
\label{int.cap}
Cap_A(U):=\inf\{\EE(w,w): w\ge \mathbf1_U\,\, m\mbox{-a.e.},\,\, w\in D(\EE)\},
\end{equation}
and for arbitrary $B\subset E$, $Cap_A(B)=\inf\{Cap_A(V): B\subset V,\,\,\, V\,\mbox{is open}\}$. 
Using this subadditive set function, we may consider a unique decomposition:
\[
\mu=\mu_d+\mu_c
\]
of any measure $\mu\in\MM_\rho$, where $\mu_c\bot Cap_A$ and $\mu_d\ll Cap_A$.
It is well known (see \cite{KR:NoDEA2}) that if $\mu\ll Cap_A$, then there exists a solution to \eqref{eq1.1}.
However, if  $\mu_c$ is non-trivial, then by \cite[Theorem 4.14]{BMP}, one can find  a function $f$ (independent of $x\in E$)  satisfying the above conditions
such that there is no solution to \eqref{eq1.1} with $A=\Delta$ on a bounded domain
$D\subset \BR^d$ with zero Dirichlet  boundary condition.  This shows that the presence of the non-trivial concentrated part $\mu_c$ of measure $\mu$
 changes the picture completely. Our goal is to study the existence and non-existence mechanism for \eqref{eq1.1} hidden in the relation
between operator $A$, right-hand side $f$ and concentrated measure $\mu_c$.

It is worth mentioning here  that, by \cite[Proposition 4.3, Theorem 4.9]{K:NoDEA}, if $u$ is a  solution to \eqref{eq1.1}
with $\mu\in\MM_1$, then
it is a {\em renormalized solution} to \eqref{eq1.1}, i.e.
\begin{enumerate}[(i)]
\item $f(\cdot,u)\in L^1(E;m)$, $T_k(u):=\max\{-k,\min\{u,k\}\}\in D_e(\EE),\, k\ge 0$, 
\item for any $k\ge 0$ there exists a bounded  measure $\lambda_k$ 
such that $\lambda_k\ll Cap_A$ and  for any bounded $\eta \in D_{e}(\EE)$
\[
\EE(T_k(u),\eta)=\int_Ef(\cdot,u)\eta\,dm+\int_E\tilde\eta\,d\mu_d+\int_E\tilde\eta\,d\lambda_k,
\]
\item $\int_E\xi\,d\lambda_k\to\int_E\xi\,d\mu_c$ as $k\to \infty$ for any $\xi\in C_b(E)$ (i.e. {\em narrowly}).
\end{enumerate}
Here $\tilde\eta$ is a {\em quasi-continuous} $m$-version of $\eta$, and $(\EE,D_e(\EE))$
is the extended Dirichlet form  (see Section \ref{sec2}). The converse is also true - renormalized solution solves \eqref{eq1.1} - 
if we assume additionally e.g., that the resolvent of $(A,D(A))$ maps $\BB_b(E)$ into $C_b(E)$ (see \cite[Theorem 4.9]{K:NoDEA}).

{\bf Main results of the paper.}  Fix  a strictly positive  bounded Borel function $\varrho\in L^1(E;m)\cap L^2(E;m)$ such that 
\[
R\varrho:= \int_E G(\cdot,y)\varrho(y)\,m(dy)\le \rho\quad m\mbox{-a.e.},
\] 
and $R\varrho$ is bounded.  For the existence results, we need one more assumption on $f$, namely that 
 for any $\underline u,\overline u\in L^1(E;\varrho\cdot m)$ such that $f(\cdot,\underline u), f(\cdot,\overline u)\in L^1(E;\rho\cdot m)$ we have
\[
x\longmapsto \sup_{y\in [\underline u(x),\overline u(x)]}|f(x,y)|\in L^1(E;\rho\cdot m).
\]
The above condition is satisfied,   supposing e.g. that   $f$ is non-increasing with respect to $y$
or there exists an increasing function $g$ on $\BR$ such that $c_1 g(y)\le |f(x,y)|\le c_2 g(y),\, x\in E,y\in\BR$
for some $c_1,c_2>0$.
In Section \ref{sec4}, we extend  Perron's method of supersolutions and subsolutions
and prove the following result (Theorem \ref{th3.1}).
\begin{center}
\begin{minipage}[c][1,65cm][t]{0,92\textwidth}
\textbf{Theorem 1.}  
{\em Let $A,f,\mu$ be as in the foregoing. Assume that there exists a subsolution $\underline u$ to (\ref{eq1.1}) and a supersolution $\overline u$ to (\ref{eq1.1})
such that $\underline u\le \overline u,\,m$-a.e. Then, there exists  a maximal solution $u$ to (\ref{eq1.1}) such that $\underline u\le u\le \overline u,\, m$-a.e.
}
\end{minipage}
\end{center}
The above result, but for $A=\Delta$, was proved in \cite{MP}. 
To go further in Section \ref{sec5}, we extend  the theory of {\em reduced measures} introduced by Brezis, Marcus and Ponce in \cite{BMP1,BMP}
for the classical Dirichlet Laplacian and monotone $f$.  Let $\GG(f)$ be a class of measures $\mu\in\MM_\rho$
for which there exists a solution to \eqref{eq1.1}, and for given $\nu\in\MM_\rho$ denote by $\GG_{\le\nu}(f)$ (resp. $\GG_{\ge\nu}(f)$)
the class of measures $\mu\in\GG(f)$ satisfying  $\mu\le\nu$ (resp. $\mu\ge\nu$). The measures belonging to $\GG(f)$
are called {\em good measures}. As we mentioned before in general $\GG(f)\subsetneq \MM_\rho$.
It appears that for any $\mu\in \MM_\rho$ such that $\GG_{\le\mu}(f)\neq\emptyset$, we can always find the biggest measure $\mu^{*,f}$ less than $\mu$
such that $\mu^{*,f}\in\GG(f)$ - such measure is called a {\em reduced measure}. This is, in particular, the content of our second main result 
(Theorem \ref{th4.1}).
 
\begin{center}
\begin{minipage}[c][5,65cm][t]{0,92\textwidth}
\textbf{Theorem 2.}  
{\em Let $A,f,\mu$ be as in the foregoing. 
Assume that $\GG_{\le\mu}(f)\neq \emptyset$. 
\begin{enumerate}[(1)] 
\item There exists $\mu^{*,f}\in \GG_{\le \mu}(f)$
such that
\[
\mu^{*,f}=\max \GG_{\le \mu}(f).
\]
\item Let $\phi$ be a strictly positive function in $L^1(E,\rho\cdot m)$. For any $n\ge 1$ there exists  a maximal solution $u_n$ to
\begin{equation*}
-Au=\max\{-n\phi,f\}(\cdot,u)+\mu.
\end{equation*}
Moreover,  $u_n\searrow u^{*,f}$, where $u^{*,f}$ is a maximal solution to
\begin{equation*}
-Au=f(\cdot,u)+\mu^{*,f}.
\end{equation*}
\item $\mu^{*,f}=\mu_d-\mu^-_c+\nu$ for a Borel measure $\nu$ on $E$ satisfying $0\le\nu\le \mu^+_c$.
\end{enumerate}
}
\end{minipage}
\end{center}

Clearly, when  $\mu$ is a good measure, then $\mu=\mu^{*,f}$ and so, by assertion (2) of the above theorem,
for any $\mu\in\GG(f)$ there exists a maximal solution to \eqref{eq1.1}.  For brevity, and when there is no risk for confusion, we mostly  omit the superscript $f$ on $\mu^{*,f}$ and $u^{*,f}$.  In case $\GG_{\ge \mu}(f)\neq\emptyset$, we may  also consider the problem of the existence of  the smallest good measure 
 greater than $\mu$. By a simple calculation we find that
\[
\mu_{*,f}:=-(-\mu)^{*,\tilde f},
\]
where $\tilde f(x,y):= -f(x,-y),\, x\in E,\, y\in\mathbb R$, is a solution to the latter problem. Therefore, by Theorem 2 provided $\GG_{\ge\mu}(f)\neq\emptyset$, we get the existence of $\mu_{*,f}\in \GG(f)$ such that
\[
\mu_{*,f}=\min\GG_{\ge\mu}(f).
\]
Using the notion of reduced measures and
the results of Theorem 1 and Theorem 2, we easily get  the following result (we follow the idea of A.C. Ponce).

\begin{center}
\begin{minipage}[c][1,25cm][t]{0,92\textwidth}
\textbf{Theorem 3.}  
{\em Let $A,f,\mu$ be as in the foregoing. 
Assume that there exists a subsolution  and a supersolution to (\ref{eq1.1}).
 Then there exists  a maximal solution to (\ref{eq1.1}).
}
\end{minipage}
\end{center}
Observe that contrary to  Theorem 1, we do not demand  in Theorem 3 that the subsolution be less than or equal to the supersolution $m$-a.e.

In Section \ref{sec6}, we prove a series of results concerning the properties of the set $\GG(f)$
and the reduction operator $\mu\longmapsto \mu^*$. They exhibit that the properties of these 
two mathematical objects, proved before in the literature  for Dirichlet Laplacian and non-increasing $f$, extend to 
the general framework considered here with an exception that the reduction operator is no longer Lipschitz continuous.
Therefore, the main concern of Section \ref{sec6} will be continuity of the reduction operator. We also observe that
 certain mapping defined via the reduction operator is a continuous metric projection onto $\GG(f)$. 

\begin{center}
\begin{minipage}[c][5,65cm][t]{0,92\textwidth}
\textbf{Theorem 4.}  
{\em Let $A,f,\mu$ be as in the foregoing. 
The set $\GG(f)$ is a convex and closed  subset of $\MM_\rho$ with total variation norm $\|\mu\|_{\rho}:= \int_E\rho\,d|\mu|$.
Moreover, the mapping
\[
\Pi_{f}:\MM_\rho\to \GG(f)
\]
defined as $\Pi_{f}(\mu):= (\mu^+)^*+(-\mu^-)_*$ is a continuous metric projection onto $\GG(f)$, i.e.
\[
\|\Pi_{f}(\mu)-\mu\|_\rho=\inf_{\nu\in\GG(f)} \|\nu-\mu\|_\rho.
\]
Moreover, if $Q:\MM_\rho\to\GG(f)$ is a metric projection onto $\GG(f)$, with a property that for any  orthogonal $\mu,\nu\in\MM_\rho$,
\[
Q(\mu+\nu)=Q(\mu)+Q(\nu),
\] 
then $Q=\Pi_{f}$.
}
\end{minipage}
\end{center}

One of the illustrative results concerning the structure of the set $\GG(f)$, easily following from the results of Section \ref{sec6}, is the following equality
\[
\GG(f)=\mathcal A(f)+L^1(E;\rho\cdot m),
\]
where $\mathcal A(f)$ is the class of {\em admissible measures}: it  consists of measures $\mu\in\MM_\rho$
such that $|f(\cdot,R\mu)|\in L^1(E;\rho\cdot m)$.
In Section \ref{sec7}, we  prove  much stronger result. Let $B_{L^1}(0,r):=\{u\in L^1(E;\rho\cdot m): \|u\|_{L^1(E;\rho\cdot m)}\le r\}$.

\begin{center}
\begin{minipage}[c][3cm][t]{0,92\textwidth}
\textbf{Theorem 5.}  
{\em 
 Let $A,f,\mu$ be as in the foregoing. Moreover, assume that $\rho$ is bounded and there exists $\varepsilon>0$ such that $\sup_{|y|\le\varepsilon}|f(\cdot,y)|\in L^1(E;\rho\cdot m)$. 
\begin{enumerate}
\item[(1)] For any $r>0$
\[
\GG(f)=\mathcal A(f)+B_{L^1}(0,r).
\]
\item[(2)] Let $cl$ denote the closure operator  in $(\MM_\rho,\|\cdot\|_\rho)$. Then
\[
\GG(f)=cl \mathcal A(f).
\]
\end{enumerate}
}
\end{minipage}
\end{center}
The last assertion implies that for any $g$, satisfying the same conditions as $f$,  if
\[
f(u)\sim g(u),\quad |u|\to \infty,
\]
then $\mu^{*,f}=\mu^{*,g}$ for positive $\mu\in\MM_\rho$ (see Corollary \ref{cor.asym12}). 
In other words, the reduction operator and the class of good measures, for $f$ independent of the state variable,
 depend only on the behavior  of $f$ at infinity.

{\bf Some comments on the  literature related to the problem}. 
Concerning the existence results for \eqref{eq1.1} with $\mu\ll Cap_A$, we mention the paper by H. Brezis and W. Strauss \cite{BS}, and by Y. Konishi \cite{Konishi},
where $f$ is assumed to be non-increasing and independent of $x\in E$ and $\mu\in L^1(E;m)$,
the paper by T. Klimsiak and A. Rozkosz \cite{KR:JFA}, where $f$ is non-increasing, and their another paper 
 \cite{KR:MM}, where $f$ merely satisfies  the sign condition.

As to the existence results for \eqref{eq1.1} with general measure data, above all, B\'enilan and Brezis' paper \cite{BB}
should be mentioned. It was published in 2004, however it summarizes, among other things, the existence and non-existence results
on the problem \eqref{eq1.1}, with $A=\Delta$ and non-increasing $f$, achieved in the period  1975-2004 (see Appendix A in \cite{BB}).
Most part of the said  paper is concerned with variational problems related to the Thomas-Fermi energy functional:
\begin{align*}
J_{TF}(\eta):= \frac12 \int_E R\eta\cdot\eta\,dm+\int_E(j(\cdot,\eta)-R\mu)\,dm,
\end{align*}
where $j:E\times\BR\to[0,\infty]$ is a function satisfying
\begin{itemize}
\item $j(x,0)=0\,\,m$-a.e., $j(x,r)=\infty,\, r<0\,\, m$-a.e., $j(x,r)<\infty,\, r>0\,\, m$-a.e.
\item $r\longmapsto j(x,r)$ is convex and l.s.c. $m$-a.e.,
\end{itemize}
and the domain of $J_{TF}$ is as follows
\[
D(J_{TF}):=\{\eta\in L^1(E;m): \eta\ge 0,\, \int_ER\eta\cdot\eta\,dm<\infty,\, j(\cdot,\eta)-R\mu\in L^1(E;m)\}.
\]
However, in \cite{BB} it is given an interesting result which relates  problem \eqref{eq1.1} with 
Thomas-Fermi functional. Namely, under some additional assumptions on $R$, it is proved in \cite[Theorem 1]{BB} that if a strictly positive  $\eta_0\in D(J_{TF})$
($I:=\int_E\eta_0\,dm$)  satisfies 
\begin{equation}
\label{eq.mp1}
J_{TF}(\eta_0)\le J_{TF}(\eta),\quad \eta\in D(J_{TF}),\, \int_E\eta\,dm=I,
\end{equation}
then there exists $\lambda\in \BR$ such that
\begin{equation}
\label{eq.mp12}
-Au=f_\lambda(\cdot, u)+\mu,
\end{equation}
with  $u:= R\eta_0-R\mu$ and $f_\lambda(x,y):=-(\partial j)^{-1}(x,y-\lambda)$, where $\partial j$
is the subdifferential of $j$ with respect to the second variable. Although the paper \cite{BB} is focused on
the minimization problem \eqref{eq.mp1}, which leads to \eqref{eq.mp12} in the special case of $\eta_0$
being strictly positive, it is worth mentioning that for the existence of $\eta_0$ in \eqref{eq.mp1} it is always assumed in \cite{BB}
that, up to translation, $\mu\in\mathcal A(f_0)$ (see conditions (H), (H$^+$), (3.18) in \cite{BB}).

In 2004 Brezis, Marcus and Ponce \cite{BMP1,BMP} introduced the notion of {\em reduced measures} for \eqref{eq1.1}
with $A=\Delta$ and $f$ being non-increasing. Since then the research on equations of the form \eqref{eq1.1} has
flourished once more, mostly with $A$ being the Dirichlet Laplacian or Dirichlet (rarely regional) fractional Laplacian.
We limit ourselves to mentioning \cite{BLO,CFY,DPP,MP,Wang,Ve} in case of Dirichlet Laplacian or divergence form diffusion operators and  \cite{CV1,CV2,CY,KMS,LDH}
in case of the fractional Laplacian. 

The theory of reduced measures was generalized by the author of the present paper in \cite{K:CVPDE} to
a class of Dirichlet operators and  with $f$ being non-increasing. The goal of the present paper is to analyze 
equation \eqref{eq1.1} for the same class of operators as considered in \cite{K:CVPDE} but under the  assumption
that $f$ merely satisfies the sign condition \eqref{eq2.1.1}. Some results, however, are new even for monotone $f$, e.g.
 Theorem 5.

\section{Notation, basic notions and standing assumptions}
\label{sec2}
As it was said  in the introduction $A$ is assumed to be a  self-adjoint operator on $L^2(E;m)$
generating a strongly continuous Markov semigroup $(T_t)_{t\ge 0}$ on $L^2(E;m)$ - so called {\em Dirichlet operators}.
Thus,   $(T_t)_{t\ge 0}$ is a contraction on $L^2(E;m)$ and, as a result, $(0,\infty)$
is a subset of the resolvent set for $(T_t)_{t\ge 0}$. By $(J_\alpha)_{\alpha>0}$ we denote the resolvent family for $(T_t)_{t\ge 0}$ 
on $L^2(E;m)$.  Throughout the paper, we assume that $(T_t)_{t\ge 0}$ is {\em transient}, i.e. there exists a strictly positive
function $g\in L^2(E;m)$ such that 
\[
J_0g:=\esssup_{n\ge 1} \int_0^n T_tg\,dt<\infty\quad m\mbox{-a.e.}
\]
Moreover, we assume that there exists the Green function $G$ for $-A$. The precise meaning of this condition
shall be explained in Section \ref{sec3}. 

In what follows we fix a strictly positive {\em excessive function} (see Section \ref{sec2.2}) $\rho$ and strictly positive  bounded Borel measurable  function $\varrho\in L^1(E;m)\cap L^2(E;m)$ such that $J_0\varrho$ is bounded and 
\begin{equation}
\label{eq2.1}
J_0\varrho\le \rho\quad m\text{-a.e.}
\end{equation}
For the existence of $\varrho$ see e.g. \cite[Lemma 6.1]{K:arxiv}.  Observe that  if $m(E)<\infty$ and  $E$ is Green bounded, i.e.
\[
\sup_{x\in E}\int_{E}G(x,y)\,m(dy)<\infty,
\]
then we may take $\varrho\equiv\rho\equiv const$.
We also consider  the following condition
\begin{enumerate}
\item[(A1)] for any $\underline u,\overline u\in L^1(E;\varrho\cdot m)$ such that $\underline u\le \overline u$ and  $f(\cdot,\underline u), f(\cdot,\overline u)\in L^1(E;\rho\cdot m)$ we have
\[
x\mapsto \sup_{y\in [\underline u(x),\overline u(x)]}|f(x,y)|\in L^1(E;\rho\cdot m).
\]
\end{enumerate}
Observe that (A1) is easily verified provided $f$ is non-increasing with respect to $y$
or there exists an increasing function $g:\BR\to \BR$ such that $c_1 g(y)\le |f(x,y)|\le c_2 g(y),\, x\in E,y\in\BR$
for some $c_1,c_2>0$.
By $\MM_\rho$ we denote the set of Borel measures on $E$ such that 
\[
\|\mu\|_{\rho}:= \int_E\rho\,d|\mu|<\infty.
\]
$\BB(E)$ stands for the set of Borel measurable function on $E$, and $\BB_b(E)$ (resp. $\BB^+(E)$)
is a subset of $\BB(E)$ consisting of bounded (resp. positive) functions.
For given Borel measure $\mu$ on $E$ and $\eta\in\BB(E)$ such that $\int_E|\eta|\,d|\mu|<\infty$ we let
\[
\langle\eta,\mu\rangle:=\int_E\eta\,d\mu.
\]
We also denote by $\eta\cdot\mu$ a Borel measure on $E$ defined as follows
\[
\langle \xi,\eta\cdot\mu\rangle:= \langle\xi\eta,\mu\rangle,\quad \xi\in\BB_b(E).
\]
By $(\EE,D(\EE))$, we denote a  symmetric Dirichlet form on $L^2(D;m)$ generated by $(A,D(A))$ defined as follows:
$D(\EE):= D(\sqrt{-A})$, and $\EE(u,v):= (\sqrt{-A}u,\sqrt{-A}v),\, u,v\in D(\EE)$. Throughout the paper, we assume that
$(\EE,D(\EE))$ is {\em regular}, i.e. $C_c(E)\cap D(\EE)$ is dense in $D(\EE)$ with the norm $\|\cdot\|_{\EE_1}:=\big(\EE(\cdot,\cdot)+(\cdot,\cdot)_{L^2(E;m)}\big)^{1/2}$ and in $C_c(E)$ with the uniform convergence norm. Self-adjoint Dirichlet operators with regular associated form $(\EE,D(\EE))$
shall be called regular.

\subsection{Elements of potential theory}
\label{sec2.1}
Let us remind that $Cap$ is a set function defined by \eqref{int.cap}.
We say that a property $P$ holds {\em quasi-everywhere} (q.e. for short) on $E$ if it holds except a set $B\subset E$
such that $Cap(B)=0$. 
An increasing sequence $\{F_n\}$ of closed subsets of $E$ is called a {\em nest} iff $Cap(E\setminus F_n)\rightarrow 0,\, n\rightarrow \infty$.
A function  $u$ on $E$ is called  {\em quasi-continuous} iff  for any $\varepsilon>0$ there exists closed set $F_\varepsilon\subset E$
such that $Cap(E\setminus F_\varepsilon)\le\varepsilon$ and $u_{|F_\varepsilon}$ is continuous.
By \cite[Theorem 2.1.2]{FOT}, $u$ is quasi-continuous if and only if there exists a nest $\{F_n\}$ such that  for any $n\ge 1$, $u_{|F_n}$ is continuous.
An increasing sequence $\{F_n\}$ of closed subsets of $E$ is called a {\em generalized nest} iff for every compact $K\subset E$,
$Cap(K\setminus F_n)\rightarrow 0,\, n\rightarrow \infty$. A  Borel measure $\mu$ on $E$ is called {\em smooth} iff it is absolutely continuous with respect to $Cap$, and 
there exists a generalized nest $\{F_n\}$ such that $|\mu|(F_n)<\infty,\, n\ge 1$. 
$\MM_{\rho}^0$ stands for a subset of $\MM_\rho$ consisting of smooth measures.
We say that  a measurable function $u$ on $E$ is {\em quasi-integrable} iff for every $\varepsilon>0$ there exists a closed set $F_\varepsilon\subset E$ such that $Cap(E\setminus F_\varepsilon)\le\varepsilon$ and $\mathbf{1}_{F_\varepsilon}u\in L^1(E;m)$. We say that a measurable function $u$ on $E$ is {\em locally quasi-integrable} iff for every compact $K\subset E$, $\mathbf{1}_{K}u$ is quasi-integrable.

\begin{proposition}
\label{prop2.2}
A measurable function $u$ on $E$ is locally quasi-integrable iff  there exists a generalized nest $\{F_n\}$ such that $\mathbf{1}_{F_n}u\in L^1(E;m),\, n\ge 1$.
\end{proposition}
\begin{proof}
Sufficiency.  Consider a compact set  $K\subset E$. 
We shall show that $\mathbf1_K u$ is quasi-integrable. Fix  $\varepsilon>0$. 
Let $V$ be a relatively compact open set such that $K\subset V$.
By the assumption  there exists a closed set $F_{n_\varepsilon}$ such that $Cap(\overline V\setminus F_{n_\varepsilon})\le \varepsilon$
and $\mathbf{1}_{F_{n_\varepsilon}}u\in L^1(E;m)$. Set $F_\varepsilon= V^c\cup F_{n_\varepsilon}$. Then $\mathbf{1}_{F_\varepsilon\cap K}u\in L^1(E;m)$ and
\[
Cap(E\setminus F_\varepsilon)= Cap (V\setminus F_{n_\varepsilon})\le Cap (\overline V\setminus F_{n_\varepsilon})\le \varepsilon.
\]
Necessity.  Let $\{E_n\}$ be an increasing sequence of relatively compact open sets such that $\bigcup_{n\ge 1} E_n=E$.
By the assumption for every $n\ge 1$ there exists closed $F_{k_n}\subset E_n$ such that $Cap(E_n\setminus F_{k_n})\le \frac{1}{n}$
and $\mathbf{1}_{F_{k_n}}u\in L^1(E;m)$. Set
\[
F_n=\bigcup_{j=1}^n F_{k_j}.
\]
Clearly $\{F_n\}$ is an increasing sequence of closed sets and $\mathbf{1}_{F_n}u\in L^1(E;m),\, n\ge 1$.  Let $K\subset E$
be a compact set. Then there exists $n_0\ge 1$ such that $K\subset E_{n},\, n\ge n_0$.  For $n\ge n_0$,
\[
Cap(K\setminus F_n)\le Cap (E_n\setminus F_{k_n})\le \frac1n.
\]
Therefore, $\{F_n\}$ is a generalized nest. 
\end{proof}
Since we assumed that $(T_t)$ is transient, there exists a strictly positive function $g$ on $E$
such that
\[
\int_E|u|\,g\,dm\le \big(\EE(u,u)\big)^{1/2},\quad u\in D(\EE).
\]
Therefore, there exists the  extension $D_e(\EE)$ of $D(\EE)$ such that: $D_e(\EE)\subset L^1(E;g\cdot m)$,
$D(\EE)=D_e(\EE)\cap L^2(E;m)$, and for any $u\in D_e(\EE)$, there exists a sequence $\{u_n\}\subset D(\EE)$ that 
is a Cauchy sequence in the norm $\|\cdot\|_{\EE}$ and satisfies $u_n\to u$ in $L^1(E;g\cdot m)$ and $\EE(u_n,u_n)\to \EE(u,u)$  (see \cite[Theorem 1.5.1, Theorem 1.5.2]{FOT}). Clearly, $(\EE,D_e(\EE))$ is a Hilbert space. By \cite[Theorem 2.1.7]{FOT}, any $u\in D_e(\EE)$
possesses an $m$-version which is quasi-continuous. In what follows for $u\in D_e(\EE)$ we denote by $\tilde u$ quasi-continuous
$m$-version of $u$.

\subsection{Probabilistic potential theory}
\label{sec2.2}
Let $\partial$ be either an isolated point added to $E$ - provided $E$ is compact - or
a one-point compactification of $E$ - provided $E$ is not compact. 
We let $E_\partial:= E\cup\{\partial\}$.
Throughout the paper, we adopt the 
convention that whenever $f$ is a function defined on $B\subset E$, then it is automatically extended  to $B\cup\{\partial\}$
by putting $f(\partial)=0$. Let
\[
\Omega:=\{\omega:[0,\infty)\to E_\partial: \omega\text{  is c\'adl\'ag, and   }\omega(s)=\partial,\, s\ge t \text{  whenever  }\omega(t)=\partial\}.
\]
Recall that a function  $\omega: [0,\infty)\to E_\partial$ is called {\em c\'adl\'ag} if it is right-continuous on $[0,\infty)$
and left-limited on $(0,\infty)$. We endow $\Omega$ with the Skorohod topology $d$. Then $(\Omega,d)$ is a separable metric space (see \cite[Section 12]{bil}). 
We also consider {\em shift operators} $(\theta_t)_{t\ge 0}$:
\[
\theta_t:\Omega\to \Omega,\quad \theta_t(\omega)(s):=\omega(s+t),\quad s,t\ge 0,
\]
and a family of {\em projection operators} (also called the {\em canonical process}) $X_t:\Omega\to E_{\partial},\, t\ge 0$, $X_t(\omega):= \omega(t),\, \omega\in \Omega$.
 Let $\FF^0_t:=\sigma(X_s^{-1}(B): s\le t,\, B\in\BB(E_\partial))$. By \cite[Theorem 7.2.1]{FOT}, there exists a family $(P_x)_{x\in E_\partial}$
 of Borel probability measures on $\Omega$ and a right-continuous filtration $(\FF_t)_{t\ge 0}$ on $\Omega$
 such that $\mathbb X:=\big((P_x)_{x\in E_\partial},(\FF_t)_{t\ge 0}\big)$ is a {\em Hunt process} on $E_\partial$
 associated with $(A,D(A))$, i.e. for any $f\in\BB_b(E)\cap L^2(E;m)$,
 \[
 T_tf(x)=\int_\Omega f(X_t(\omega))\,P_x(d\omega),\quad t\ge 0,\, m\text{-a.e.}
 \]
 By the very definition of a Hunt process, $\FF^0_t\subset \FF_t,\, t\ge 0$.
 The question of uniqueness of $\mathbb X$ is treated in \cite[Theorem 4.2.8]{FOT}.
 Let $\zeta$ stand for the lifetime of process $\mathbb X$, i.e.
 \[
 \zeta(\omega):=\inf\{t\ge 0: X_t(\omega)=\partial\}.
 \]
 We let $\FF_\infty:= \sigma(\FF_t: t\ge 0)$. For $t\in [0,\infty]$, we denote by $b\FF_t$ a set of bounded real valued $\FF_t$
 measurable functions. In what follows, we consider the following notation
 \[
\mathbb E_xF:=  \int_\Omega F(\omega)\,P_x(d\omega),\quad F\in b\FF_\infty
 \]
We define for  any $f\in \BB^+(E)$,
\[
P_tf(x):=\mathbb E_xf(X_t),\quad R_\alpha f(x):=\mathbb E_x\int_0^\infty e^{-\alpha s} f(X_s)\,ds,\quad t\ge 0,\,\alpha\ge 0,\, x\in E.
\] 
We let $R:=R_0$. We say that a  property $P$ holds {\em almost surely} (a.s.) (resp. {\em quasi almost surely} (q.a.s.)) on $\Omega$
if it holds $P_x$-a.s. for any $x\in E$ (resp. for q.e. $x\in E$). A Borel measurable positive function $f$ on $E$ is called {\em $\alpha$-excessive},
where $\alpha\ge 0$, if 
\[
\sup_{t>0}e^{-\alpha t}P_tf(x)=f(x),\quad x\in E.
\]
In case $\alpha=0$, we just say that $f$ is excessive.

Recall that a family $(A_t)_{t\ge 0}$ of $\mathbb R\cup\{+\infty\}$ valued functions on $\Omega$
is called an {\em additive functional of}  $\mathbb X$ if there exist  $\Lambda\in\FF_\infty$
and $N\subset E$ such that
\begin{enumerate}
\item[(1)] $\theta_t(\Lambda)\subset \Lambda,\, t\ge 0$, $Cap(N)=0$, $P_x(\Lambda)=1,\, x\in E\setminus N$,
\end{enumerate}
furthermore, for any $\omega\in\Lambda$,
\begin{enumerate}
\item[(2)] $A_{t+s}(\omega)=A_s(\omega)+A_{t}(\theta_s\omega),\, s,t\ge 0$,
\item[(3)] $|A_t(\omega)|<\infty,\, t\in [0,\zeta(\omega))$,
\item[(4)] $t\longmapsto A_t(\omega)$ is c\'adl\'ag on $[0,\zeta(\omega))$,
\item[(5)] $A_t(\omega)=A_{\zeta(\omega)}(\omega),\, t\ge\zeta(\omega)$,
\item[(6)] for any $t\ge 0$, $A_t$ is $\FF_t$-measurable, and $A_0(\omega)=0$.
\end{enumerate}
$\Lambda$ is called a {\em defining set} of $(A_t)$, and $N$ is called an {\em exceptional set} of $(A_t)$.
An additive functional (AF for short) $(A_t)$  of $\mathbb X$ is said to be  {\em positive} if $A_t(\omega)\ge 0,\, t\ge 0,\, \omega\in\Lambda$.
We say that an AF $(A_t)$ of $\mathbb X$ is continuous if $t\longmapsto A_t(\omega)$ is continuous on $[0,\infty)$
for any $\omega\in \Lambda$.  In what follows, we frequently use the notion of positive continuous additive functionals (PCAF for short) of $\mathbb X$.
We say that $(A_t)$ is a {\em  martingale additive functional}  (MAF for short) of $\mathbb X$ if it is an AF of $\mathbb X$ and
an $(\FF_t)$-martingale under measure $P_x$ for any $x\in E\setminus N$.  
Analogously, we say that $(A_t)$ is a {\em  local MAF}  of $\mathbb X$ if it is an AF of $\mathbb X$, and
it is a local $(\FF_t)$-martingale under measure $P_x$ for any $x\in E\setminus N$. By \cite[Theorem 5.1.4]{FOT} there is a ono-to-one 
correspondence between PCAFs of $\mathbb X$ and positive smooth measures - so called Revuz duality.
PCAF $(A_t)$ of $\mathbb X$ and positive smooth measure $\nu$ on $E$ are in Revuz duality if for any positive $f\in\BB(E)$,
\[
\mathbb E_x\int_0^\infty f(X_r)\,dA_r=\int_EG(x,y)f(y)\,\nu(dy),\quad x\in E\setminus N.
\]
By \cite[Theorem 5.1.3, Theorem 5.1.4]{FOT}, there exists a unique PCAF of $\mathbb X$ satisfying the above identity for any $f\in \BB^+(E)$.
We shall denote it by $A^\nu$. We say that $\nu$ is {\em strictly smooth} if the exceptional set $N$ of $A^\nu$
is empty.

\begin{proposition}
\label{prop2.1}
A positive Borel measure $\mu$ on $E$ is smooth iff it is absolutely continuous with respect to $Cap$ and  there exists a strictly positive quasi-continuous function $u$ on $E$ such that $\int_Eu\,d\mu<\infty$.
\end{proposition}
\begin{proof}
Sufficiency. Since $u$ is quasi-continuous there exists a nest $\{F_n\}$ such that $u_{|F_n}$ is quasi-continuous for every $n\ge 1$.
Let $\{E_n\}$ be an increasing  sequence of relatively compact open  sets such that $\bigcup_{n\ge 1}E_n=E$. Set $\tilde F_n=\overline E_n\cap F_n$. Obviously, $\{\tilde F_n\}$ is a generalized nest. Moreover, by continuity of $u_{|F_n}$ and compactness of $\overline E_n$,
$\inf_{\tilde F_n}u=c_n>0$. Thus
\[
\mu(\tilde F_n)\le \frac{1}{c_n}\int_E u\,d\mu<\infty,\quad n\ge 1.
\]
Necessity.  Let $\phi$ be a strictly positive bounded   function such that $R\phi$ is bounded (see \cite[Corollary 1.3.6]{Oshima}) and $\phi\in L^1(E;m)$.
Set 
\[
\eta(x):=\mathbb E_x\int_0^\zeta \phi(X_r)e^{-A^\mu_r}\,dr,\quad x\in E.
\]
By \cite[Lemma 5.1.5]{FOT}, $\eta, R\phi$ are quasi-continuous, and
\[
\mathbb E_x\int_0^\zeta \eta(X_r)\,dA^\mu_r\le R\phi(x),\quad \mbox{q.e.}
\]
From this and \cite[Theorem 5.1.3]{FOT},
\[
\int_E \eta R\phi\,d\mu=\int_E \Big(\mathbb E_x\int_0^\zeta \eta (X_r)\,dA^\mu_r\Big)\phi(x)\,m(dx)\le \langle R\phi,\phi\rangle\le \|R\phi\|_\infty\|\phi\|_{L^1(E;m)}.
\]
Function $u=\eta R\phi$ fullfils  the requirements. 
\end{proof}

\begin{corollary}
\label{cor2.1}
Let $u$ be a measurable function on $E$. Then $u$ is locally quasi-integrable iff
\[
P_x(\int_0^t|u(X_r)|\,dr<\infty,\, t<\zeta)=1,\quad \mbox{q.e.}
\]
\end{corollary}
\begin{proof}
Follows from \cite[Theorem 5.1.4]{FOT} and Proposition \ref{prop2.2}.
\end{proof}

\section{Green's functions}
\label{sec3}
We say that  a Borel measurable function $G:E\times E\to \mathbb R^+\cup\{+\infty\}$
is the Green function for $-A$ if
\[
Rf(x)=\int_E G(x,y)f(y)\,m(dy),\quad x\in E,\, f\in\BB^+(E),
\]
and  $G(x,\cdot), G(\cdot,y)$ are excessive for any $x,y\in E$.
By \cite[Theorem 4.2.4]{FOT}, there exists the Green function for $-A$ if and only if $P_tf(x)=0$
for any $x\in E$ and $t>0$ provided $f\in\BB^+(E)$ and $\int_Ef\,dm=0$. Furthermore, by \cite[Lemma 4.2.4]{FOT}, there exists the Green function for $-A$ if and only if $R_\alpha f(x)=0$ for any $x\in E$ and $\alpha>0$ provided $f\in\BB^+(E)$ and $\int_Ef\,dm=0$.

At this point, we would like to formulate a general  condition
guaranteeing the existence of Green's function.  For $f\in\BB(E)$, we let $|f|_\infty:= \sup_{x\in E}|f(x)|$.

\begin{proposition}
\label{prop.xeg}
Assume that $E$ is complete. Furthermore, suppose that 
\begin{enumerate}[(i)]
\item  $T_t(C_c(E))\subset C_b(E),\, t>0$;
\item $(J_\alpha)_{\alpha>0}$ is strongly Feller: for some $\alpha>0$ (hence for any $\alpha>0$)
$J_\alpha(L^2(E;m)\cap L^\infty(E;m))\subset C_b(E)$;
\item there exists a set $\mathcal C\subset \{f\in D(A)\cap C_b(E): Af \text{ is bounded}\}$  such that any function in $C_c(E)$ is the limit in the supremum norm of functions 
from $\mathcal C$. 
\end{enumerate}
Then there exists the Green function for $-A$.
\end{proposition}
\begin{proof}
For $f\in C_c(E)$, we set
\begin{equation}
\label{eq.xeg2}
S_tf(x):= \overline{T_t f}(x),\quad x\in E,
\end{equation}
where $\overline{T_t f}$ is  continuous $m$-version of $T_t f$. Clearly, $S_t$ is a linear operator, and for any
positive $f\in C_c(E)$, we have $S_tf(x)\ge 0,\, x\in E$. By the Riesz theorem, there exists a Borel measure $\sigma(t,x,dy)$
such that for any $f\in C_c(E)$,
\[
S_tf(x)=\int_E f(y)\sigma(t,x,dy),\quad x\in E,\, t>0.
\] 
Thus, we may extend operator $S_t$ to $S_t:\BB_b(E)\to \BB_b(E)$.  
By a standard semigroup identity, we find that for any $f\in \mathcal C$,
\[
|S_tf(x)-f(x)|\le \int_0^t\int_E\sigma(r,x,dy)|Af(y)|\,dr\le t\|Af\|_{L^\infty(E;m)},\quad t>0,\, x\in E.
\]
Now, let $f\in C_c(E)$ and let $(f_n)\subset \mathcal C$ converge uniformly to $f$. Then
\begin{align*}
|S_tf(x)-f(x)|&\le |S_tf(x)-S_tf_n(x)|+|S_tf_n(x)-f_n(x)|+|f_n(x)-f(x)|\\&
\le |S_tf_n(x)-f_n(x)|+2|f_n-f|_\infty.
\end{align*}
Letting $t\to 0^+$ and then $n\to \infty$, we obtain that for any $f\in C_c(E)$, 
\[
|S_tf-f|_\infty \to 0\quad\text{as}\,\, t\to 0^+.
\]
By \cite[Lemma 2.8, page 181]{BH}, for any $f\in C_b(E)$ and any compact $K\subset E$,
\begin{equation}
\label{eq.xeg3}
\sup_{x\in K}|S_tf(x)-f(x)| \to 0\quad\text{as}\,\, t\to 0^+.
\end{equation}
By \cite[Theorem 2.9, page 150]{Casteren}, there exists a Hunt proces $((Q_x)_{x\in E}, (\HH_t)_{t\ge 0})$ associated with the semigroup kernel $(S_t)$.
Since $S_t(C_b(E))\subset C_b(E),\, t>0$, we easily deduce that for any $\alpha>0$ and $f\in C_b(E)$,
\[
U_\alpha f(x):= \int_0^\infty e^{-\alpha t}S_tf(x)\,dt,\quad x\in E
\]
belongs to $C_b(E)$. 
From this and \eqref{eq.xeg2}, we deduce that  for any $f\in C_c(E)$, $R_\alpha f(x)=\overline{J_\alpha f}(x),\, x\in E$, where $\overline{J_\alpha f}$
is the continuous $m$-version of $J_\alpha f$. Fix $x\in E$.
By the Riesz theorem, there exists a measure $\upsilon_\alpha(x,dy)$
such that $U_\alpha f(x)=\int_E f(y)\upsilon(x,dy),\, \alpha>0,\, f\in C_b(E)$. As a result, $U_\alpha$ may be extended
as the mapping $U_\alpha:\BB_b(E)\to \BB_b(E)$. We easily find that for any $f\in\BB_b(E)\cap L^2(E;m)$
\[
R_\alpha f(x)=\overline{J_\alpha f}(x),\quad x\in E.
\]
Consequently,  we obtain that $\upsilon_\alpha(x,dy)=u_\alpha(x,y)m(dy),\, x\in E,\, \alpha>0$.
\end{proof}

For sufficient conditions ensuring  condition (i) of the above proposition see e.g. \cite{BBCK,CK,DK} 
Condition (ii) of the above proposition is satisfied, for example, if for any $t>0$ there exists $c_t>0$
such that
\begin{equation}
\label{eq.apxttt}
\|T_t f\|_{L^\infty(E;m)}\le c_t\|f\|_{L^1(E;m)},\quad f\in L^2(E;m)\cap L^1(E;m),
\end{equation}
and furthermore, there exists a dense set $\mathcal C\subset L^1(E;m)$ such that $T_t(\mathcal C)\subset C_b(E),\, t>0$
(the last condition holds whenever (i) is satisfied).
Condition \eqref{eq.apxttt} is satisfied provided there exist $\delta\ge 0$ and $\nu>0$
such that
\[
\|u\|_{L^2(E;m)}^{2+4/\nu}\le C[\EE(u,u)+\delta\|u\|^2_{L^2(E;m)}]\|u\|_{L^1(E;m)}^{4/\nu},\quad u\in D(\EE)\cap L^1(E;m).
\]
(see \cite[Theorem 2.1]{CKS}). The above condition is satisfied for  \eqref{eq1.3a0} (see e.g. \cite[Theorem 1.2]{BBCK}).

Let $\varphi: \BR^+\to\BR^+$ be a strictly increasing function  with $\varphi(0)=0$.
Now, we focus on the operator \eqref{eq1.3a0}.
Consider the following conditions:
\begin{enumerate}[(A)]
\item there exist $c_1,c_2>0$  such that
\[
c_1\le a(x,y)\le c_2,\quad x,y\in\BR^d,
\]
\item there exists $c_3>0$ such that 
\[
\int_0^r\frac{s}{\varphi(s)}\,ds\le c_3\frac{r^2}{\varphi(r)},\quad r>0,
\]
\item there exist $c_4,c_5,\delta_1, \delta_2>0$ such that
\[
c_4\Big(\frac Rr\Big)^{\delta_1}\le \frac{\varphi(R)}{\varphi(r)}\le c_5\Big(\frac Rr\Big)^{\delta_2},\quad 0<r\le R.
\]
\end{enumerate}

By  \cite[Theorem 1.2]{CK}, we get the following result.

\begin{proposition}
Let $A$ be the operator of the form \eqref{eq1.3a0}
satisfying (A)--(C). Then there exists the Green function for $-A$.
\end{proposition}

{\bf Part of $A$ on an open set $D\subset E$.} For a given self-adjoint Dirichlet operator $A$ and an open set $D\subset E$, we may define self-adjoint Dirichlet operator $A_{|D}$ as follows: let
\[
\EE_{|D}(u,v):=\EE(u,v),\quad u,v\in D(\EE_{|D}):=\{w\in D(\EE): \tilde w=0\,\,\text{q.e. on }\, E\setminus D\}.
\]
By \cite[Theorem 4.4.3]{FOT}, $\EE_{|D}$ is a Dirichlet form, and if $\EE$ is regular, then $\EE_{|D}$ is regular  too. Therefore, by \cite[Theorem 1.3.1, Theorem 1.4.1]{FOT},
there exists a unique self-adjoint Dirichlet operator $(B,D(B))$ on $L^2(D;m)$ ($\subset L^2(E;m)$) such that $D(B)\subset D(\EE_{|D})$,
and 
\[
(-Bu,v)_{L^2(D;m)}=\EE_{|D}(u,v),\quad u\in D(B),\, v\in D(\EE_{|D}).
\]
We let $A_{|D}:= B$. The operator $A_{|D}$ is called a {\em part of $A$ on $D$} (or restriction of $A$ to $D$).
This operation on a Dirichlet operator $A$ is often used  when approaching the Dirichlet problem on $D$
for $A$.

{\bf Perturbation of $A$ by a smooth measure.} Let $\nu$ be a positive smooth measure on $E$, and  $A$ be a  regular self-adjoint Dirichlet operator on $L^2(E;m)$.
Define
\[
\EE_\nu(u,v):=\EE(u,v)+\int_E\tilde u\,\tilde v\,d\nu,\quad u,v\in D(\EE_\nu):=\{w\in D(\EE): \tilde w\in L^2(E;\nu)\}.
\]
By \cite[Theorem IV.4.4]{MR}, $\EE_\nu$ is a symmetric Dirichlet form on $L^2(E;m)$. Therefore, there exists a unique 
self-adjoint operator $B$ on $L^2(E;\nu)$ such that $D(B)\subset D(\EE_\nu)$, and
\[
(-Bu,v)_{L^2(E;m)}=\EE_\nu(u,v),\quad u\in D(B),\, v\in D(\EE_\nu).
\]
We set $A_\nu:= B$. Formally, $-A_\nu=-A+\nu$, so  $-A_\nu$ may be called a perturbation of $-A$ be the measure $\nu$.

{\bf Resurrected (regional)  operator.}  Let $A$ be a regular self-adjoint Dirichlet operator on $L^2(E;m)$ with
regular symmetric Dirichlet form $(\EE,D(\EE))$. By the  Beurling-Deny formulae (see \cite[Theorem 3.2.1]{FOT}) for any $u,v\in C_c(E)\cap D(\EE)$,
\begin{equation}
\label{B-D.eq}
\EE(u,v)=\EE^{(c)}(u,v)+\int_{E\times E\setminus diag} (u(x)-u(y))(v(x)-v(y))\,J(dx,dy)+\int_E u\,v\, d\kappa,
\end{equation}
where  $\mathcal{E}^{(c)}$ is a symmetric form, with domain $D(\EE^{(c)})= D(\EE)\cap C_{c}(E)$, which  satisfies the strong local property:
\[
\EE^{(c)}(u, v)=0 \quad \text { for } u,v \in D(\EE^{(c)})
\]
such that $v$ is constant on a neighbourhood of $\text{supp}[u]$.
$J$ is a symmetric positive Radon measure on $E \times E\setminus diag$ and $\kappa$ is a positive Radon measure on $E$.
Such $\mathcal{E}^{(c)}, J, \kappa$ are uniquely determined by $(\mathcal{E},D(\EE))$. 
Let $\EE^{res}(u,v):=\EE(u,v)-\int_Euv\,d\kappa$. By \cite[Theorem 5.2.17]{CF}, $(\EE^{res}, D(\EE^{res}))$ is a regular symmetric Dirichlet form
on $L^2(E;m)$ with $D(\EE^{res})$ described in \cite[(5.2.25)]{CF}. Therefore, there exists a unique self-adjoint Dirichlet operator $(B,D(B))$
such that $D(B)\subset D(\EE^{res})$ and
\[
(-Bu,v)=\EE^{res}(u,v),\quad u\in D(B),\, v\in D(\EE^{res}).
\] 
We let $A^{res}:=B$. Operator $A^{res}$ is very useful for interpretation of 
the Neumann problem  on $D$ for purely jumping Dirichlet operators $A$ (see e.g. \cite{Abatangelo}).
Indeed, one takes $(A_{|D})^{res}$ (for the fractional Laplacian, we then derive so called regional fractional Laplacian).

\begin{proposition}
Let $(A, D(A))$ be a regular self-adjoint Dirichlet operator on $L^2(E;m)$. 
Suppose that there exists the Green function $G$ for $-A$.
\begin{enumerate}[(1)]
\item If $D$ is an open subset of $E$, then there exists the Green function for $-A_{|D}$.
\item If $\nu$ is a strictly smooth positive measure on $E$, then there exists the Green function for $-A_\nu$.
\item If the killing measure $\kappa$ from decomposition \eqref{B-D.eq} is strictly smooth, then there exists the Green function
for $-A^{res}$.

\end{enumerate}
\end{proposition}
\begin{proof}
(1) It follows from \cite[Theorem 4.4.2]{FOT}. For (2), see \cite[Theorem A.2.12]{FOT}. Ad (3).
Let $\phi(x)=1,\, x\in E$. Clearly, $\phi(X_t)=\mathbf1_{\{t<\zeta\}}$. By the comment preceding Lemma 5.3.3 in \cite{FOT},
\[
\phi(X_t)-\phi(X_0)=-A^{\kappa}_t+M_t,\quad t\ge 0
\]
is the Doob-Meyer decomposition of supermartingale $\phi(X)$, where $A^{\kappa}$ is a PCAF associated with
killing measure $\kappa$ appearing in the Beurling-Deny formulae \eqref{B-D.eq}, and $M$ is a martingale additive functional of $\mathbb X$.
By \cite[Theorem 62.19]{Sharpe} a family of measures $(Q_x)$ given by
\begin{equation}
\label{eq.qq}
Q_x(F\mathbf1_{\{T<\zeta\}}):=P_x(Fm_T),\quad F\in b\FF_T, \,x\in E,\, T\ge 0,
\end{equation}
constitutes a Right Markov process on $E$, where 
\[
m_t:=\frac{\phi(X_t)}{\phi(X_0)}e^{A^{\kappa}_t}=\phi(X_t)e^{A^{\kappa}_t}=\mathbf1_{\{t<\zeta\}}e^{A^{\kappa}_t},\quad t\ge 0.
\]
By \cite[Theorem 5.2.17]{CF}, $((Q_x), X)$ is a Hunt process associated with $\EE^{res}$. By \eqref{eq.qq}, $Q_x$ is equivalent to $P_x$
for any $x\in E$. Thus, there exists the Green function for $-A^{res}$.
\end{proof}

\section{Method of sub- and supersolutions}
\label{sec4}

In what follows $f:E\times\BR\to \BR$. Consider the following conditions.
 
\begin{enumerate}
\item[Car)] $f$ is a Carath\'eodory function, i.e.
\begin{itemize}
\item $x\mapsto f(x,y)$ is Borel measurable for any $y\in\mathbb R$,
\item $y\mapsto f(x,y)$ is continuous for $m$-a.e. $x\in E$;
\end{itemize}  
\item[Sig)] for any $u\in \BR$,  $f(x,u)u\le 0\,\, m\mbox{-a.e.}\,\, x\in E$;
\item[Int)]  for any $\underline u,\overline u\in L^1(E;\varrho\cdot m)$ such that $f(\cdot,\underline u), f(\cdot,\overline u)\in L^1(E;\rho\cdot m)$ we have
\[
x\longmapsto \sup_{y\in [\underline u(x),\overline u(x)]}|f(x,y)|\in L^1(E;\rho\cdot m);
\]
\item[qM)] for any $M>0$ the mapping $E\ni x\longmapsto \sup_{|y|\le M} |f(x,y)|$ is locally quasi-integrable.
\item[M)] for any $M>0$ the mapping $E\ni x\longmapsto \sup_{|y|\le M} |f(x,y)|\in L^1(E;\rho\cdot m)$.

\end{enumerate}

For any $\mu\in\MM_\rho$, we let $\mu_d$ denote the part of $\mu$, which is  absolutely continuous with respect to $Cap$,
and by $\mu_c$ we denote the part of $\mu$, which is orthogonal to $Cap$. Observe that $\mu_d$ is a smooth measure. Indeed, 
by the very definition $\mu_d\ll Cap$. Now,
let $g\in L^2(E;m)\cap\BB(E)$ be strictly positive. We have 
\[
\int_E R_1(\rho\wedge g)\,d|\mu_d|\le \int_E\rho\,d|\mu|<\infty.
\]
We used here the fact that $\rho$ is excessive. Clearly, $\eta:=R_1(\rho\wedge g)$ is strictly positive.
By \cite[Theorem 4.2.3]{FOT}, $\eta$ is quasi-continuous. Therefore, by Proposition \ref{prop2.1}, $\mu_d$ is smooth.

\begin{remark}
\label{rem3.1}
Throughout the paper, we frequently use  without special mentioning the following facts.

(a) If $\mu\in\MM_\rho$, then
 $R|\mu|<\infty$ q.e.  and, as a result, $R|\mu|$ is quasi-continuous. The last assertion  follows from  \cite[Theorem 3.1]{K:CVPDE}.
At the same time,
\[
\int_ER|\mu|\varrho\,d m= \int_E R\varrho\,d|\mu|\le \int_E\rho\,d|\mu|<\infty.
\] 
Thus, since $\varrho$ is strictly positive, we obtain that  $R|\mu|<\infty$ $m$-a.e. By \cite[Theorem A.2.13(v)]{CF}, we have that, in fact,  $R|\mu|<\infty$ q.e. (see \cite[Lemma 2.1.4]{FOT}).

(b) If $u_1, u_2$ are quasi-continuous functions on $E$, then $u_1\le u_2\,\,m$-a.e. if and only if $u_1\le u_2$ q.e.

(c) Let $B\in\BB(E)$.  If $Cap(B)=0$, then $P_x(\exists_{t\ge 0}: X_t\in B)=0$ q.e. (see \cite[Theorem 4.2.1]{FOT}).
\end{remark}

Recall that a measurable function $\tau:\Omega\to [0,\infty]$ is called a {\em stopping time} if $\{\omega\in\Omega: \tau(\omega)\le t\}\in\FF_t$
for any $t\ge 0$.
We say that a non-decreasing sequence $\{\tau_k\}$ of {\em stopping times} is a {\em reducing sequence} for a measurable function $u$
on $E$ if $\tau_k\wedge\zeta\to \zeta$ q.a.s., and 
\[
\mathbb E_x\sup_{t\le\tau_k}|u(X_t)|<\infty,\quad k\ge 0\,\,\, \mbox{q.e.}
\] 

\begin{lemma}
\label{lm.m}
Let $\mu\in\MM_\rho$. Set $w:=R|\mu|$, $u(x):= R\mu(x),\, x\in E\setminus N$ and zero on $N$, where $N=\{x\in E: w(x)=\infty\}$.
Then, there exists a local MAF $M$ such that, for any $x\in E\setminus N$,
\begin{equation}
\label{eq.lmm1}
u(X_t)=u(X_0)-\int_0^t\,dA^{\mu_d}_r+\int_0^t\,dM_r,\quad t\ge 0\,\,\, P_x\mbox{-a.s.}
\end{equation}
Moreover, for $\tau_k:=\inf\{t>0: w(X_t)\ge k\}\wedge k$, we have
\begin{equation}
\label{eq.lmm2}
\mathbb E_x\sup_{t\le \tau_k}|u(X_t)|+\mathbb E_x\sup_{t\le\tau_k}|M_t|<\infty,\quad k\ge 0,\, x\in E\setminus N.
\end{equation}
In particular, $\{\tau_k\}$ is a reducing sequence for $u$.
\end{lemma}
\begin{proof}
By \cite[Theorem 3.7]{K:CVPDE}, we obtain \eqref{eq.lmm1}. From the fact that $w(X)$ is a positive supermartingale under the measure
$P_x$, for any $x\in E\setminus N$ (see e.g. \cite[Theorem II.2.12]{BG}) and  from \cite[Theorem 51.1]{Sharpe}, we infer \eqref{eq.lmm2}.
The last assertion is obvious.
\end{proof}

\begin{definition}
Let $\mu\in \MM_\rho$.
We say that a measurable function $u$ is a solution to (\ref{eq1.1}) if $f(\cdot,u)\in L^1(E;\rho\cdot m)$
and 
\begin{equation}
\label{eq3.00cb1}
u(x)=\int_E G(x,y) f(y,u(y)) \,m(dy)+\int_E G(x,y)\,\mu(dy),\quad m\mbox{-a.e.}
\end{equation}
\end{definition}

\begin{remark}
\label{rem3.1addfc}
By Remark \ref{rem3.1}, $R|f(\cdot,u)|<\infty$ q.e., $R|\mu|<\infty$ q.e., and  $Rf(\cdot,u)$, $R\mu$
are quasi-continuous. As a result, one sees that a solution to (\ref{eq1.1}) has always a quasi-continuous $m$-version. Observe also that
by (\ref{eq2.1}), $u\in L^1(E;\varrho\cdot m)$.
\end{remark}

\begin{proposition}
Assume that $\mu\in\MM_1$, and $R_1(\BB_b(E))\subset C_b(E)$ or  $u \in L^1(E;m)$, $R_1(C_b(E))\subset C_b(E)$. Then $u$ is a solution to \eqref{eq1.1} if and only if $u$ is a renormalized solution to \eqref{eq1.1}:
\begin{enumerate}[(i)]
\item $T_k(u):=\min\{k,\max\{u,-k\}\}\in D_e(\EE),\, k>0$;
\item $f(\cdot,u)\in L^1(E;m)$;
\item There exists a family $(\nu_k)_{k\ge 0}\subset \MM_1$ of smooth measures such that
$\nu_k\to \mu_c$ in the narrow topology, as $k\to\infty$;
\item For any bounded $\eta\in D(\EE)$,
\[
\EE(T_k(u),\eta)=\langle f(\cdot,u),\eta\rangle+\langle \mu_d,\tilde \eta\rangle+\langle\nu_k,\tilde\eta\rangle.
\]
\end{enumerate}
\end{proposition}
\begin{proof}
It follows from \cite[Theorem 4.9]{K:NoDEA}.
\end{proof}

\begin{definition}
\label{def.subsup}
Let $\mu\in \MM_\rho$.
We say that a measurable function $u$ is a subsolution (resp. supersolution) to (\ref{eq1.1}) if $f(\cdot,u)\in L^1(E;\rho\cdot m)$
and  there exists a positive measure $\nu\in \MM_\rho$ such that
\begin{equation}
\label{eq3.00cb2}
u(x)=\int_E G(x,y) f(y,u(y)) \,m(dy)+\int_E G(x,y)\,\mu(dy)-\int_E G(x,y)\,\nu(dy),\quad m\mbox{-a.e.}
\end{equation}
(resp.)
\begin{equation}
\label{eq3.00cb3}
u(x)=\int_E G(x,y) f(y,u(y)) \,m(dy)+\int_E G(x,y)\,\mu(dy)+\int_E G(x,y)\,\nu(dy),\quad m\mbox{-a.e.}
\end{equation}
\end{definition}

Throughout the paper, unless stated otherwise,  we always consider quasi-continuous $m$-versions of solutions, supersolutions and subsolutions to (\ref{eq1.1}). These versions may be defined as right-hand sides of \eqref{eq3.00cb1}, \eqref{eq3.00cb2} or \eqref{eq3.00cb3}, where finite,  and zero otherwise.

\begin{proposition}
\label{prop3.1}
Let $u,w$ be  subsolutions to (\ref{eq1.1}). Then $u\vee w$ is a subsolution to (\ref{eq1.1}). 
\end{proposition}
\begin{proof}
By the definitions of sub- and supersolution to (\ref{eq1.1}) there exist positive $\nu^1,\nu^2\in \MM_\rho$
such that
\begin{equation}
\label{eq3.00}
u=Rf(\cdot,u)+R\mu-R\nu^1,\quad w=Rf(\cdot,u)+R\mu-R\nu^2,\quad\mbox{q.e.}
\end{equation}
By Lemma \ref{lm.m} there exist local MAFs $M,N$ such that
\[
u(X_t)=u(X_0)-\int_0^t f(X_r,u(X_r))\,dr-\int_0^t\,dA^{\mu_d}_r+\int_0^t \,dA^{\nu^1_d}_r+\int_0^t\,dM_r,\quad  t\ge 0,
\]
\[
w(X_t)=w(X_0)-\int_0^t f(X_r,w(X_r))\,dr-\int_0^t\,dA^{\mu_d}_r+\int_0^t \,dA^{\nu^2_d}_r+\int_0^t\,dN_r,\quad t\ge 0,
\]
q.a.s. By the Tanaka-Meyer formula  (see, e.g., \cite[IV.Theorem 70]{Protter}) there exists an increasing c\`adl\`ag process $C$, with $C_0=0$, such that 
\begin{align*}
(u\vee w)(X_t)&=(u\vee w)(X_0)-\int_0^t \mathbf{1}_{\{u\ge w\}}(X_r)f(X_r,u(X_r))\,dr-\int_0^t \mathbf{1}_{\{u\ge w\}}(X_r)\,dA^{\mu_d}_r\\&\quad
+\int_0^t \mathbf{1}_{\{u\ge w\}}(X_r)\,dA^{\nu^1_d}_r+\int_0^t \mathbf{1}_{\{u\ge w\}}(X_{r-})\,dM_r
\\&\quad -\int_0^t \mathbf{1}_{\{w> u\}}(X_r)f(X_r,w(X_r))\,dr-\int_0^t \mathbf{1}_{\{w>u\}}(X_r)\,dA^{\mu_d}_r\\&\quad+\int_0^t \mathbf{1}_{\{w>u\}}(X_r)\,dA^{\nu^2_d}_r
+\int_0^t \mathbf{1}_{\{w>u\}}(X_{r-})\,dN_r+ \int_0^t\,dC_r,\quad t\ge 0,\quad\mbox{q.a.s.}
\end{align*}
Hence,
\begin{align}
\label{eq3.0}
\nonumber
(u\vee w)(X_t)&=(u\vee w)(X_0)-\int_0^t f(X_r,(u\vee w)(X_r))\,dr-\int_0^t \,dA^{\mu_d}_r\\&\quad
\nonumber
+\int_0^t \mathbf{1}_{\{u\ge w\}}(X_r)\,dA^{\nu^1_d}_r+\int_0^t \mathbf{1}_{\{w>u\}}(X_r)\,dA^{\nu^2_d}_r+ \int_0^t\,dC_r\\&\quad
+\int_0^t \mathbf{1}_{\{u\ge w\}}(X_{r-})\,dM_r
+\int_0^t \mathbf{1}_{\{w>u\}}(X_{r-})\,dN_r,\quad t\ge 0,\quad\mbox{q.a.s.}
\end{align}
From the above formula, we deduce that $C$ is an additive functional of $\mathbb X$. Thus, $C$ is a positive AF
of $\mathbb X$.
Moreover, since $\mathbb X$ is a Hunt process and $u,w$ are quasi-continuous, $u(X), w(X), M,N$
have only totally inaccessible jumps (see \cite[Theorem 4.2.2, Theorem A.2.1, Theorem A.3.6]{FOT}). Therefore, dual predictable projection $\tilde C$ of $C$ is continuous.
By \cite[Theorem A.3.16]{FOT}, $\tilde C$ is a PCAF. Consequently, by \cite[Theorem 5.1.4]{FOT}, there exists a positive smooth measure $\beta$
such that $\tilde C=A^\beta$. Let $\{\tau_k\}$ be a reducing sequence for $u\vee w$. Then, by (\ref{eq3.0}),
\begin{align}
\label{eq3.4}
\nonumber
(u\vee w)(x)&=\mathbb E_x(u\vee w)(X_{\tau_k})+\mathbb E_x\int_0^{\tau_k} f(X_r,(u\vee w)(X_r))\,dr+\mathbb E_x\int_0^{\tau_k} \,dA^{\mu_d}_r\\&\quad
-\mathbb E_x\int_0^{\tau_k} \mathbf{1}_{\{u\ge w\}}(X_r)\,dA^{\nu^1_d}_r-\mathbb E_x\int_0^{\tau_k} \mathbf{1}_{\{w>u\}}(X_r)\,dA^{\nu^2_d}_r-\mathbb E_x\int_0^{\tau_k}\,dA^\beta_r,\,\mbox{q.e.}
\end{align}
By \cite[Theorem 3.7]{K:CVPDE},
\[
\mathbb E_xu(X_{\tau_k})\rightarrow R\mu_c(x)-R\nu^1_c(x),\quad \mathbb E_xw(X_{\tau_k})\rightarrow R\mu_c(x)-R\nu^2_c(x),\quad\mbox{q.e.}
\]
Moreover, by \cite[Theorem 3.7, Theorem 6.3]{K:CVPDE},
\[
\mathbb E_x|u(X_{\tau_k})-w(X_{\tau_k})|\rightarrow R|\nu^1_c-\nu^2_c|(x),\quad \mbox{q.e.}
\]
Thus,
\begin{align*}
\lim_{k\rightarrow \infty}\mathbb E_x(u\vee w)(X_{\tau_k})&=\lim_{k\rightarrow \infty}\frac12 \Big( \mathbb E_xu(X_{\tau_k})+\mathbb E_xw(X_{\tau_k})+\mathbb E_x|u(X_{\tau_k})-w(X_{\tau_k})|\Big)\\&
=R\mu_c(x)-\frac12R(\nu^1_c+\nu^2_c-|\nu^1_c-\nu^2_c|)(x)=R\mu_c(x)-R(\nu^1_c\wedge\nu^2_c)(x),\quad\mbox{q.e.}
\end{align*}
Therefore, letting $k\rightarrow \infty$ in (\ref{eq3.4}), we get
\[
u\vee w=Rf(\cdot,u\vee w)+R\mu-R(\mathbf{1}_{\{u\ge w\}}\cdot\nu^1_d)-R(\mathbf{1}_{\{w>u\}}\cdot\nu^2_d)-R\beta- R(\nu^1_c\wedge \nu^2_c),\quad\mbox{q.e.}
\]
From the fact that $u\le u\vee w$ and \cite[Lemma 4.6]{K:CVPDE}, we infer that $\beta\in\MM_\rho$.
Therefore, $u\vee w$ is a subsolution to \eqref{eq1.1}.
\end{proof}

\begin{proposition}
\label{prop3.2}
Assume \rm{\text{Car)}}. Let $\mu\in \MM_{\rho}$. Suppose that there exist positive $g\in L^1(E;\rho\cdot m)$ and $c>0$ such that
\begin{equation}
\label{eq3.1}
|f(x,y)|\le cg(x),\quad x\in E,\, y\in \BR.
\end{equation}
Then there exists a solution to (\ref{eq1.1}).
\end{proposition}
\begin{proof}
Set $r:= c\|g\|_{L^1(E;\rho\cdot m)}+\|\mu\|_{\MM_\rho}$. For $u\in L^1(E;\varrho\cdot m)$, we let 
\[
\Phi(u)=Rf(\cdot,u)+R\mu.
\]
Observe that by (\ref{eq3.1})
\begin{equation}
\label{eq3.2}
|\Phi(u)|\le cRg+R|\mu|.
\end{equation}
Hence
\[
\|\Phi(u)\|_{L^1(E;\varrho\cdot m)}\le r.
\]
By (\ref{eq3.1}) and Car), we get easily that $\Phi$ is continuous. Let $\{u_n\}\subset L^1(E;\varrho\cdot m)$.
Observe that
\[
\Phi(u_n)=Rf^+(\cdot,u_n)-Rf^-(\cdot,u_n)+R\mu
\]
By \cite[Lemma 94, page 306]{DellacherieMeyer}, $\{Rf^+(\cdot,u_n)\}, \{Rf^-(\cdot,u_n)\}$ have subsequence (still denoted by $(n)$) convergent $m$-a.e.
Thus, up to subsequence, $\{\Phi(u_n)\}$ is convergent $m$-a.e. By (\ref{eq3.2}) the Lebesgue dominated convergence theorem is applicable, and so,  $\{\Phi(u_n)\}$ is convergent
in $L^1(E;\varrho\cdot m)$. By Schauder's fixed point theorem, we get the result.
\end{proof}

\begin{proposition}
\label{prop3.3}
Assume {\rm Sig)}. Let $u$ be a subsolution to (\ref{eq1.1}) and $w$ be a supersolution to (\ref{eq1.1}). Then
\begin{enumerate}
\item[(1)] $u^++Rf^-(\cdot,u)\le R(\mathbf{1}_{\{u> 0\}}\cdot\mu^+_d)+R\mu_c^+$, q.e., and  
\[
\|f^-(\cdot,u)\|_{L^1(E;\rho\cdot m)}\le\|\mathbf{1}_{\{u> 0\}}\cdot\mu^+_d\|_{\MM_\rho}+\|\mu_c^+\|_{\MM_\rho},
\]
\item[(2)]  $w^-+Rf^+(\cdot,w)\le R(\mathbf{1}_{\{w\le 0\}}\cdot\mu^-_d)+R\mu_c^-$, q.e., and
\[
\|f^+(\cdot,w)\|_{L^1(E;\rho\cdot m)}\le\|\mathbf{1}_{\{w\le 0\}}\cdot\mu^-_d\|_{\MM_\rho}+\|\mu_c^-\|_{\MM_\rho}.
\]
\end{enumerate}
\end{proposition}
\begin{proof}
The proofs of (1) and (2) are analogous. We shall give the proof  of (1). Since $u$ is a subsolution, there exists a positive $\nu\in \MM_\rho$
such that
\[
u=Rf(\cdot,u)+R\mu-R\nu,\quad\mbox{q.e.}
\]
By Lemma \ref{lm.m} and the Tanaka-Meyer formula
\begin{align*}
u^+(x)-\mathbb E_x\int_0^{\tau_k}\mathbf{1}_{\{u>0\}}(X_r)f(X_r,u(X_r))\,dr&\le \mathbb E_xu^+(X_{\tau_k})+\mathbb E_x\int_0^{\tau_k}\mathbf{1}_{\{u>0\}}(X_r)\,dA^{\mu_d}_r\\&\quad-\mathbb E_x\int_0^{\tau_k}\mathbf{1}_{\{u>0\}}(X_r)\,dA^{\nu_d}_r,\quad\mbox{q.e.},
\end{align*}
where $\{\tau_k\}$ is a reducing sequence for $u$.
From this and Sig), we infer that
\[
u^+(x)+\mathbb E_x\int_0^{\tau_k}f^-(X_r,u(X_r))\,dr\le \mathbb E_xu^+(X_{\tau_k})+\mathbb E_x\int_0^{\tau_k}\mathbf{1}_{\{u>0\}}(X_r)\,dA^{\mu_d}_r,\quad\mbox{q.e.}
\]
Letting $k\rightarrow \infty$ and using \cite[Theorem 3.7, Theorem 6.3]{K:CVPDE} yields
\[
u^++Rf^-(\cdot,u)\le R\mu_c^++R(\mathbf{1}_{\{u>0\}}\cdot\mu_d),\quad\mbox{q.e.}
\]
By \cite[Lemma 4.6]{K:CVPDE}, we get (1).
\end{proof}

\begin{theorem}
\label{th3.1}
Let $\mu\in\MM_\rho$. Assume {\rm Car), Int)}. 
\begin{enumerate}
\item[(1)]  Let  $\psi\in L^1(E;\varrho\cdot m)$ be such that $f(\cdot,\psi)\in L^1(E;\rho\cdot m)$.
Suppose that there exists a subsolution $\underline u$ to (\ref{eq1.1}) such that $\underline u\le \psi\,\, m$-a.e.
Then there exists a maximal subsolution $u^*$ to (\ref{eq1.1}) such that $u^*\le \psi\,\, m$-a.e.
\item[(2)]  Assume that there exists a subsolution $\underline u$ to (\ref{eq1.1}) and a supersolution $\overline u$ to (\ref{eq1.1})
such that $\underline u\le \overline u,\,m$-a.e. Then there exists  a maximal solution $u$ to (\ref{eq1.1}) such that $\underline u\le u\le \overline u,\, m$-a.e. The maximal solution $u$ is at the same time a maximal subsolution lying between $\underline u$ and $\overline u\,\,m$-a.e.
\item[(3)]  Assume {\rm Sig)}. For any subsolution $\underline u$ to \eqref{eq1.1} we have $\underline u\le R\mu^+\,\, m$-a.e. 
\item[(4)]  Assume {\rm Sig)}. Let $\psi:E\to \BR\cup \{+\infty\}$ be a Borel measurable function. 
Suppose that there exists a subsolution $\underline u$ to \eqref{eq1.1} such that $\underline u\le \psi\,\,m$-a.e.
Then there exists a maximal subsolution $u^*$ to \eqref{eq1.1} such that $u^*\le \psi\,\, m$-a.e.
\end{enumerate}
\end{theorem}
\begin{proof}
Ad (1). Set
\[
\mathcal {S}_\psi=\{v: v\,\, \mbox{is a subsolution to}\,\, (\ref{eq1.1})\,\,  \mbox{and}\,\, v\le \psi\,\, m\mbox{-a.e.}\}
\]
By assumptions $\mathcal S_\psi$ is nonempty. By Remark \ref{rem3.1addfc}, $\mathcal S_\psi\subset L^1(E;\varrho\cdot m)$. Let
\[
\alpha=\sup_{v\in\mathcal S_\psi}\int_E v\varrho\,dm.
\]
By the assumptions made on $\psi$, $\alpha<\infty$. Let $\{v_n\}\subset \mathcal S_\psi$ be such that $\int_E v_n\varrho\,dm\nearrow \alpha$. Set
\[
u_n=\max\{v_1,\dots,v_n\},\quad n\ge 1.
\]
By Proposition \ref{prop3.1}, $\{u_n\}\subset \mathcal S_\psi$. Set $u^*=\sup_{n\ge 1} u_n$. 
Clearly, $\underline u\le u^*\le \psi\,\, m$-a.e. By the assumptions made on $\psi$ and Int), we have $f(\cdot,u^*)\in L^1(E;\rho\cdot m)$.
Since $u_n$ is a subsolution to (\ref{eq1.1}) there exists a positive $\nu_n\in\MM_\rho$ such that
\begin{equation}
\label{eq3.8}
u_n=Rf(\cdot,u_n)+R\mu-R\nu_n\quad\mbox{q.e.}
\end{equation}
Clearly,
\begin{equation}
\label{eq3.8burz}
u_1\le u_n\le u^*,\quad n\ge 1\,\,\,\mbox{q.e.}
\end{equation}
Hence
\[
|f(\cdot,u_n)|\le \sup_{y\in [u_1,u^*]} |f(\cdot,y)|=: g_1.
\]
By Int), $g_1\in L^1(E;\rho\cdot m)$. Therefore, by the Lebesgue dominated convergence theorem,
\[
\int_E|Rf(\cdot,u_n)-Rf(\cdot,u^*)|\varrho\,dm\le \int_E |f(\cdot,u_n)-f(\cdot,u^*)|\rho\,dm\rightarrow 0,\quad n\rightarrow \infty.
\]
Thus, up to subsequence, $\lim_{n\rightarrow\infty} Rf(\cdot,u_n)=Rf(\cdot,u^*)\,\, m$-a.e. This in turn implies that $(R\nu_n)_{n\ge 1}$ is convergent $m$-a.e.
Let $\eta:=\lim_{n\rightarrow \infty} R\nu_n\,\, m$-a.e. By \cite[Lemma 94, page 306]{DellacherieMeyer}, $\eta$ has an $m$-version (still denoted by $\eta$) such that $\eta$ is an excessive function. By  \eqref{eq3.8},\eqref{eq3.8burz}, we have
\[
R\nu_n\le R|f(\cdot,u_n)|+R|f(\cdot,u_1)|+R|\mu|+R|\nu_1| \quad\mbox{q.e.}
\]
Letting $n\rightarrow \infty$ and using \cite[page 197]{BG} yields
\[
\eta\le R|f(\cdot,u^*)|+R|f(\cdot,u_1)|+R|\mu|+R|\nu_1|.
\]
By \cite[Proposition 3.9]{GetoorGlover}, there exists a positive Borel measure $\beta$ on $E$ such that $\eta=R\beta$. From the above inequality and \cite[Lemma 4.6]{K:CVPDE}, we conclude that $\beta\in \MM_\rho$. Going back to (\ref{eq3.8}) and letting $n\rightarrow \infty$, we deduce from what has been already proven that
\[
u^*=Rf(\cdot,u^*)+R\mu-R\beta\quad\mbox{q.e.}
\]
Thus, $u^*\in \mathcal S_\psi$. What is left is to show that $u^*$ is maximal. Let $v\in \mathcal S_\psi$. 
Clearly $v\vee u_n\nearrow v\vee u^*$. By Proposition \ref{prop3.1}, $v\vee u_n\in \mathcal S_\psi$. Thus,
\[
\alpha =\lim_{n\rightarrow \infty} \int_E v_n\varrho\,dm\le \lim_{n\rightarrow \infty}\int_E u_n\varrho\,dm \le \lim_{n\rightarrow \infty}\int_Ev\vee u_n\varrho\,dm\le\alpha.
\]
By the Lebesgue monotone convergence theorem 
\[
\int_E u^*\varrho\,dm=\int_E v\vee u^*\varrho\,dm=\alpha.
\]
Therefore,
\[
\int_E(v\vee u^*-u^*)\varrho\,dm=0.
\]
Hence, $u^*\le v\vee u^*\,\, m$-a.e., which implies that $v\le u^*\,\, m$-a.e.

Ad (2). Set
\[
\hat f(x,y)=f(x,(y\wedge \overline u(x))\vee \underline u(x)),\quad x\in E,\, y\in\BR.
\]
By Int), $\hat f$ satisfies (\ref{eq3.1}) with $g(x):= \sup_{y\in [\underline u(x),\overline u(x)]}|f(x,y)|,\, x\in E$. Therefore,  there exists a solution $\hat u$ to (\ref{eq1.1}) with $f$ replaced by $\hat f$. 
Since $\underline u$ is a subsolution to (\ref{eq1.1}), there exists a positive measure $\nu\in\MM_\rho$
such that
\[
\underline u=Rf(\cdot,\underline u)+R\mu-R\nu\quad \mbox{q.e.}
\]
By Lemma \ref{lm.m} and  the Tanaka-Meyer formula
\begin{align*}
(\underline u(x)-\hat u(x))^+&\le \mathbb E_x(\underline u(X_{\tau_k})-\hat u(X_{\tau_k}))^+\\&
\quad+\mathbb E_x\int_0^{\tau_k}\mathbf{1}_{\{\underline u>\hat u\}}(X_r)(f(X_r,\underline u(X_r))-\hat f(X_r,\hat u(X_r)))\,dr
\\&\quad -\mathbb E_x\int_0^{\tau_k}\mathbf{1}_{\{\underline u>\hat u\}}(X_r) \,dA^{\nu_d}_r\quad\mbox{q.e.}
\end{align*}
Observe that, by the definition of $\hat f,\, \mathbf{1}_{\{\underline u>\hat u\}}(f(\cdot,\underline u)-\hat f(\cdot,\hat u))\le 0$. Thus,
\[
(\underline u(x)-\hat u(x))^+\le \mathbb E_x(\underline u(X_{\tau_k})-\hat u(X_{\tau_k}))^+\quad \mbox{q.e.}
\]
Letting $k\rightarrow \infty$ and using \cite[Theorem 3.6, Theorem 6.3]{K:CVPDE} yields 
\[
(\underline u(x)-\hat u(x))^+\le R(-\nu_c)^+=0\quad \mbox{q.e.},
\]
and so  $\underline u\le \hat u\,\, m$-a.e. Analogous reasoning for $\hat u, \overline u$ shows that $\hat u\le \overline u\, \,m$-a.e.
Consequently,  $\hat f(\cdot,\hat u)=f(\cdot,\hat u)\,\, m$-a.e. Therefore, $\hat u$ is, in fact, a solution to (\ref{eq1.1}) and $\underline u\le\hat u\le \overline u\,\, m$-a.e.
Now, we shall show the existence of a maximal solution to (\ref{eq1.1}) lying between $\underline u, \overline u$. 
Applying (1) with $\psi=\overline u$
gives  the existence  of a maximal subsolution $u^*$ to (\ref{eq1.1}) such that $\underline u\le u^*\le\overline u\,\,m$-a.e. By what has been already proven,
there exists a solution $\hat u$ to (\ref{eq1.1}) such that $u^*\le \hat u\le \overline u\,\, m$-a.e. On the other hand, since $\hat u$ also is 
a subsolution to (\ref{eq1.1}), $\hat u\le u^*\,\, m$-a.e. Thus, $\hat u=u^*\,\,m$-a.e. Now, we easily deduce that $u^*$ is a maximal solution to (\ref{eq1.1}).

Ad (3) Let $w=R\mu^+$, and  $v$ be a subsolution to \eqref{eq1.1}. By the definition of a subsolution to (\ref{eq1.1}) there exists a positive $\nu\in\MM_\rho$
such that
\[
v=Rf(\cdot,v)+R\mu-R\nu\quad\mbox{q.e.}
\]
By Lemma \ref{lm.m} and the Tanaka-Meyer formula
\begin{align*}
(v(x)-w(x))^+&\le \mathbb E_x(v(X_{\tau_k})-w(X_{\tau_k}))^++\mathbb E_x\int_0^{\tau_k}\mathbf{1}_{\{v>w\}}(X_r)f(X_r,v(X_r))\,dr\\&
\quad-\mathbb E_x\int_0^{\tau_k}\mathbf{1}_{\{u>w\}}(X_r)\,dA^{\mu_d^-}-\mathbb E_x\int_0^{\tau_k}\mathbf{1}_{\{u>w\}}(X_r)\,dA^{\nu_d},
\end{align*}
where $\{\tau_k\}$ is a reducing sequence for $v-w$.
Since $w$ is positive, we have, by Sig), that $\mathbf{1}_{\{v>w\}}f(\cdot,v)\le 0$. Consequently,
\[
(v(x)-w(x))^+\le \mathbb E_x(v(X_{\tau_k})-w(X_{\tau_k}))^+\quad\mbox{q.e.}
\]
By \cite[Theorem 3.7, Theorem 6.3]{K:CVPDE},
\[
\lim_{k\rightarrow \infty} \mathbb E_x(v(X_{\tau_k})-w(X_{\tau_k}))^+=R(\mu_c-\nu_c-\mu_c^+)^+(x)=0\quad\mbox{q.e.}
\]
Therefore, $v\le w\,\, m$-a.e. 

Ad (4). We maintain  the notation of the proof of (1). The proof of (4) runs exactly the same lines as the proof of (1) but with different justification
of the facts that $\alpha<\infty$ and $f(\cdot,u^*)\in L^1(E;\rho\cdot m)$.  The first property is a consequence of (3). 
For the second one, observe that, by Sig),
\begin{equation}
\label{eq3.7}
|f(\cdot,u^*)|=f^-(\cdot,u^*)+f(\cdot,-(u^*)^-).
\end{equation}
By Proposition \ref{prop3.3},
\[
\|f^-(\cdot,u_n)\|_{L^1(E;\rho\cdot m)}\le \|\mu\|_{\MM_\rho}
\]
By Fatou's lemma $f^-(\cdot,u^*)\in L^1(E;\rho\cdot m)$. At the same time, we have
\[
-u_1^-\le -(u^*)^-\le 0.
\]
Therefore by Int), $f(\cdot,-(u^*)^-)\in L^1(E;\rho\cdot m)$. Consequently,  by (\ref{eq3.7}), $f(\cdot,u^*)\in L^1(E;\rho\cdot m)$.
\end{proof}

\begin{proposition}
\label{prop3.4}
Assume {\rm Car), Int)}.
Let $\underline u$ (resp. $\overline u$) be  a subsolution (resp. supersolution) to (\ref{eq1.1}).
Let $\tilde \mu\in\MM_\rho$, $\tilde f$ be a measurable function on $E\times\BR$, and 
 $\tilde u $ be  a solution to (\ref{eq1.1}) with $f,\mu$ replaced by $\tilde f,\tilde \mu$ such that $\underline u\le \tilde u\le \overline u\,\,m$-a.e.
Let $u$ be a maximal solution to (\ref{eq1.1}) such that  $\underline u\le  u\le \overline u\,\,m$-a.e.
Assume that $\tilde f(\cdot,\tilde u)\le f(\cdot,\tilde u)\,\,m$-a.e. and   $\tilde \mu\le\mu$. 
Then $\tilde u\le u\,\, m$-a.e.
\end{proposition}
\begin{proof}
Observe that
\[
-A\tilde u=f(\cdot,\tilde u)+\mu-(f(\cdot,\tilde u)-\tilde f(\cdot,\tilde u))-(\mu-\tilde \mu).
\]
By Int), $f(\cdot,\tilde u)\in L^1(E;\rho\cdot m)$. Therefore, $\tilde u$ is a subsolution to
\[
-Au=f(\cdot,u)+\mu.
\]
By Theorem \ref{th3.1}(2), $u$ is a maximal subsolution to the above problem lying between $\underline u$ and $\overline u\,\,m$-a.e. Thus, $\tilde u\le u\,\,m$-a.e.
\end{proof}

\section{Existence of maximal and minimal good measure}
\label{sec5}

{\bf Standing assumption}: In the remainder of the paper, we assume that  conditions Car), Sig), Int) and qM)  are in force.

We begin with the following lemma which will be crucial in our proof techniques 
when passing to the limit in variety of semilinear equations.

\begin{lemma}
\label{lmj.jlm}
Assume that $\{u_n\}$ is a sequence of quasi-continuous functions on $E$, $u\in\BB(E)$
and $\{\tau_k\}$ is a non-decreasing sequence of stopping times such that $\tau_k\to \zeta\,\,$q.a.s.
Suppose that for some $p>0$ and any $k\ge 1$,
\begin{equation}
\label{eqj.j3}
\mathbb E_x\sup_{0\le t\le\tau_k}|u_n(X_t)-u(X_t)|^p\to 0,\quad\text{as}\,\,n\to \infty\,\, q.e.
\end{equation}
Then $u$ is quasi-continuous.
\end{lemma}
\begin{proof}
By \cite[Theorem 4.2.2]{FOT} process $u_n(X)$ is right-continuous on $[0,\infty)$ q.a.s.
Therefore, by \eqref{eqj.j3}, $u$ shares this property too. Consequently, by \cite[Theorem 4.6.1, Theorem A.2.7]{FOT},
$u$ is quasi-continuous.
\end{proof}

We say that $\mu\in\MM_\rho$ is a good measure iff there exists a solution to (\ref{eq1.1}).
We  let $\GG(f)$ denote the set of all good measures. It is clear that $\GG(f)$ also depends on $A$.

\begin{theorem}
\label{th4.1}
Assume that there exists a subsolution $\underline u$ to (\ref{eq1.1}).
Let $u^*$ be a maximal subsolution to (\ref{eq1.1}) (cf. Theorem \ref{th3.1}(4)). Set
\begin{equation}
\label{eq.eq.4.0}
\mu^*:= -Au^*-f(\cdot,u^*)
\end{equation}
in the sense that $\mu^*=\mu-\nu$, where $\nu$ comes from the definition of a subsolution to \eqref{eq1.1} applied to $u^*$ (see Definition \ref{def.subsup}).
Then $(\mu^*)_d=\mu_d$,  $\mu^*$ is the largest measure less then $\mu$ such that (\ref{eq1.1})
has a solution with $\mu$ replaced by $\mu^*$, and $u^*$ is a maximal solution to 
\begin{equation}
\label{eq.eq.4.0abc}
-Av=f(\cdot,v)+\mu^*.
\end{equation}
Moreover, for any $n\ge 1$ and strictly positive $\phi\in L^1(E,\rho\cdot m)$,   there exists a  maximal solution $u_n$ to
\begin{equation}
\label{eq.abcd}
-Av=\max\{-n\phi,f\}(\cdot,v)+\mu,
\end{equation}
and  $u_n\searrow u^*\,\,m$-a.e.  Furthermore, for any solution $u$ to (\ref{eq.eq.4.0abc}),
and any reducing sequence $\{\tau_k\}$ for $u$, we have
\begin{equation}
\label{eq.eq.4.1}
\mathbb E_xu(X_{\tau_k})\rightarrow R[(\mu^*)_c](x),\quad \mbox{q.e.}
\end{equation}

\end{theorem}
\begin{proof}
Let $\phi$ be a strictly positive bounded Borel function on $E$
such that $\phi\in L^1(E;\rho\cdot m)$. Set 
\[
f_n(x,y)=f(x,y)\vee (-n\phi),\quad x\in E,\, y\in\BR.
\]
Clearly, $f_n$ satisfies Sig) and $f_n\ge f_{n+1}\ge f,\, n\ge 1$.
Let $w:=R\mu^+$. Observe that
\[
0\le -f_n(\cdot,w)=f^-_n(\cdot,w)\le n\phi.
\]
Thus, $f_n(\cdot,w)\in L^1(E;\rho\cdot m)$. Moreover,
\[
-Aw=f_n(\cdot,w)+\mu+(-f_n(\cdot,w)+\mu^-).
\]
Consequently, $w$ is a supersolution to (\ref{eq.abcd}). Set $\overline u:= w$. 
Since $\underline u$ is a subsolution to (\ref{eq1.1}),
there exists a positive $\nu \in\MM_\rho$ such that
\[
-A\underline u=f(\cdot,\underline u)+\mu-\nu.
\]
Therefore,
\[
-A\underline u=f_n(\cdot,\underline u)+\mu-(f_n(\cdot,\underline u)-f(\cdot,\underline u))-\nu.
\]
Hence, $\underline u$ is a subsolution to (\ref{eq.abcd}). By Theorem \ref{th3.1}(3), $\underline u\le\overline u\,\,m$-a.e.
Consequently, by Theorem \ref{th3.1}, there exists a maximal solution  $u_n$ to \eqref{eq.abcd}
such that $\underline u\le u_n\le\overline u\,\,m$-a.e. 
By Proposition \ref{prop3.4}, $u_n\ge u_{n+1},\, n\ge 1$ q.e. Put $u=\inf_{n\ge 1} u_n$ q.e.
By Proposition \ref{prop3.3}, 
\begin{equation}
\label{eq4.0}
|u_n|+R|f_n(\cdot, u_n)|\le R|\mu|,\quad \mbox{q.e.},\qquad \|f_n(\cdot,u_n)\|_{L^1(E;\rho\cdot m)}\le \|\mu\|_{\MM_\rho}.
\end{equation}
Therefore, by  Fatou's lemma and Car),
\begin{equation}
\label{eq4.01}
|u|+R|f(\cdot, u)|\le R|\mu|,\quad \mbox{q.e.},\qquad \|f(\cdot,u)\|_{L^1(E;\rho\cdot m)}\le \|\mu\|_{\MM_\rho}.
\end{equation}
Let $\{\delta_k\}$ be a common reducing sequence for $\underline u,\overline u$. Let
\[
\sigma_k=\inf\{t\ge 0: |\underline u(X_t)|+|\overline u(X_t)|\ge k\}\wedge k,\quad \hat\sigma_{k,j}=\inf\{t\ge 0: \int_0^t \sup_{|y|\le k}|f(X_r,y)|\,dr\ge j\}.
\]
By the definition of a reducing sequence $\lim_{k\rightarrow \infty}\delta_k\wedge\zeta=\zeta$. Since $\underline u,\overline u$ are quasi-continuous,
$\lim_{k\rightarrow \infty}\sigma_k\wedge\zeta=\zeta$ (see \cite[Theorem 4.2.2]{FOT}). Finally, by qM) and Corollary \ref{cor2.1}, $\lim_{j\rightarrow \infty}\hat \sigma_{k,j}\wedge\zeta=\zeta$. Let $\tau_{k,j}:=\delta_k\wedge\sigma_k\wedge\hat\sigma_{k,j}$.
Observe that 
\begin{equation}
\label{eq.jotk}
\tau_{k,j}\nearrow \tau_k:=\delta_k\wedge\sigma_k,\quad j\to\infty.
\end{equation}
By Lemma \ref{lm.m},
\begin{equation}
\label{eq4.1j1}
u_n(X_{t\wedge\tau_{k,j}})=\mathbb E_x\big[u_n(X_{\tau_{k,j}})+\int_{t\wedge\tau_{k,j}}^{\tau_{k,j}}f_n(X_r,u_n(X_r))\,dr+\int_{t\wedge\tau_{k,j}}^{\tau_{k,j}}\,dA^{\mu_d}_r\big|\FF_{t\wedge\tau_{k,j}}\big]\quad\mbox{q.a.s.}
\end{equation}
By \cite[Lemma 6.1]{BDHPS}, for any $q\in (0,1)$, there exists $c_q>0$ such that 
\begin{align*}
\big(\mathbb E_x\sup_{0\le t\le\tau_{k,j}}|u_n(X_t)-u_l(X_t)|^q\big)^{1/q}\le c_q\mathbb E_x&\big[|u_n(X_{\tau_{k,j}})-u_l(X_{\tau_{k,j}})|\\&
+\int_0^{\tau_{k,j}}|f_n(X_r,u_n(X_r))-f_l(X_r,u_l(X_r))|\,dr\big]\quad\text{q.e}.
\end{align*}
Due to the choice of $\{\tau_{k,j}\}$ and Remark \ref{rem3.1}(c), the right-hand side of the above equation tends to zero as $n,l\to \infty$. 
Consequently, by Lemma \ref{lmj.jlm}, $u$ is quasi-continuous.  Taking $t=0$ in \eqref{eq4.1j1}, we get
\begin{equation}
\label{eq4.1}
u_n(x)=\mathbb E_xu_n(X_{\tau_{k,j}})+\mathbb E_x\int_0^{\tau_{k,j}}f_n(X_r,u_n(X_r))\,dr+\mathbb E_x\int_0^{\tau_{k,j}}\,dA^{\mu_d}_r\quad\mbox{q.e.}
\end{equation}
Letting $n\rightarrow \infty$ and using properties of $\{\tau_{k,j}\}$ and Remark \ref{rem3.1}(c)  yields
\begin{equation}
\label{eq4.2}
u(x)=\mathbb E_xu(X_{\tau_{k,j}})+\mathbb E_x\int_0^{\tau_{k,j}}f(X_r,u(X_r))\,dr+\mathbb E_x\int_0^{\tau_{k,j}}\,dA^{\mu_d}_r\quad\mbox{q.e.}
\end{equation}
Applying \eqref{eq4.01}, \eqref{eq.jotk}  and the fact that $\{\tau_k\}$ is a reducing sequence for $u$,
we find, by letting $j\to\infty$ in \eqref{eq4.2}, that
\begin{equation}
\label{eq4.2prim}
u(x)=\mathbb E_xu(X_{\tau_k})+\mathbb E_x\int_0^{\tau_k}f(X_r,u(X_r))\,dr+\mathbb E_x\int_0^{\tau_k}\,dA^{\mu_d}_r\quad\mbox{q.e.}
\end{equation}
Now, we  shall show that there exists $\beta\in\MM_\rho$ such that $u=R\beta$. For this, observe that
\begin{equation}
\label{eq4.3}
u_n=v_n-w_n\quad q.e.,
\end{equation}
where $v_n=Rf_n^+(\cdot,u_n)+R\mu^+$ and $w_n=Rf_n^-(\cdot,u_n)+R\mu^-$. Clearly, $v_n, w_n$ are excessive functions.
Moreover, by (\ref{eq4.0})
\begin{equation}
\label{eq4.4}
v_n\le 2R|\mu|,\quad w_n\le 2R|\mu|\quad \mbox{q.e.}
\end{equation}
Therefore, by \cite[Lemma 94, page 306]{DellacherieMeyer} there exists a
subsequence (still denoted by $(n)$) and excessive functions $v,w$ such that $v_n\rightarrow v$ and $w_n\rightarrow w\,\,m$-a.e.
By \cite[Proposition 3.9]{GetoorGlover}, there exist positive Borel measures $\beta_1,\beta_2$ such that $v=R\beta_1$, $w=R\beta_2$. By (\ref{eq4.4}) and \cite[Lemma 4.6]{K:CVPDE}, $\beta_1,\beta_2\in\MM_\rho$. Set $\beta=\beta_1-\beta_2$. Due to  (\ref{eq4.3}),
and the fact that $u, R\beta$ are quasi-continuous, see Remark \ref{rem3.1}, we get $u=R\beta$ q.e. Consequently, by \cite[Theorem 3.7]{K:CVPDE}
\[
\mathbb E_x u(X_{\tau_k})\rightarrow R\beta_c(x),\quad\mbox{q.e.}
\]
Therefore, letting $k\rightarrow \infty$ in (\ref{eq4.2prim}) and using (\ref{eq4.01}) yields
\begin{equation}
\label{eq4.5}
u=Rf(\cdot,u)+R\mu_d+R\beta_c,\quad\mbox{q.e.}
\end{equation}
Since $u\le u_n$, we get, by the inverse maximum principle (see \cite[Theorem 6.1]{K:CVPDE}) that $\beta_c\le \mu_c$. 
Observe that
\[
-Au=f(\cdot,u)+\mu-(\mu_c-\beta_c).
\]
Thus, $u$ is a subsolution to (\ref{eq1.1}), which in turn implies that  $u\le u^*$. On the other hand, by Proposition \ref{prop3.4}, $u^*\le u_n,\, n\ge 1$.
Thus $u=u^*$. Consequently, $\mu^*=\mu_d+\beta_c$, and  so $(\mu^*)_d=\mu_d$.  What is left is to show that $\mu^*$ is the maximal measure less then $\mu$
for which there exists a solution to (\ref{eq1.1}) with $\mu$ replaced by $\mu^*$. Let $\gamma\in\MM_\rho$,  $\gamma\le \mu$
and  $v$ be a solution to
\[
-Av=f(\cdot,v)+\gamma.
\]
Since $\gamma\le\mu$, we have that $\gamma_d\le \mu_d$ and $\gamma_c\le \mu_c$. We have already proved that $(\mu^*)_d=\mu_d$.
So that, we only  have to  prove that $\gamma_c\le (\mu^*)_c$. 
Since $\gamma\le \mu$, $v$ is a subsolution to (\ref{eq1.1}). Thus, $v\le u^*$. By the inverse maximum principle $\gamma_c\le (\mu^*)_c$.
\end{proof}

Analogous reasoning, but for supersolutions, leads to the following result.

\begin{theorem}
\label{th4.2}
Assume that there exists a supersolution  $\overline  u$ to (\ref{eq1.1}).
Let $u_*$ be a minimal supersolution to (\ref{eq1.1}). Set
\[
\mu_*:= -Au_*-f(\cdot,u_*).
\]
Then $(\mu_*)_d=\mu_d$, $\mu_*$ is the smallest measure greater  then $\mu$ such that (\ref{eq1.1})
has a solution with $\mu$ replaced by $\mu_*$, and $u_*$ is a minimal solution to 
\begin{equation}
\label{eq.eq.4.0abcmin}
-Av=f(\cdot,v)+\mu_*.
\end{equation}
Moreover, for any $n\ge 1$ and strictly positive $\phi\in L^1(E,\rho\cdot m)$,   there exists a  minimal solution $u_n$ to
\[
-Av=\min\{n\phi,f\}(\cdot,v)+\mu,
\]
and  $u_n\nearrow u_*$.
\end{theorem}

\begin{remark}
\label{rem5.1}
By Theorems \ref{th4.1}, \ref{th4.2} if $\mu\in \GG(f)$, then $u^*$ is a maximal solution (subsolution) to (\ref{eq1.1}) and $u_*$ is a minimal solution (supersolution) to (\ref{eq1.1}).   
\end{remark}

\begin{proposition}
\label{prop.cop}
Let $\mu_1,\mu_2\in\MM_\rho$ and $f_1,f_2$ satisfy Car), Sig), Int) and qM). Assume that $\mu_1\le\mu_2$ and $f_1\le f_2$.
Let $u_1, u_2$ be solutions to
\[
-Av=f_1(\cdot,v)+\mu_1,\quad -Av=f_2(\cdot,v)+\mu_2,
\]
respectively.
\begin{enumerate}
\item[(1)] If $u_2$ is maximal and $f_2(\cdot,h)\in L^1(E;\rho\cdot m)$ for some $h\in L^1(E;\varrho\cdot m)$
that satisfies $h\le u_1\,\,m$-a.e., then $u_1\le u_2\,\,m$-a.e.
\item[(2)] If $u_1$ is minimal and $f_1(\cdot,h)\in L^1(E;\rho\cdot m)$ for some $h\in L^1(E;\varrho\cdot m)$
that satisfies $u_2\le h\,\,m$-a.e., then $u_1\le u_2\,\,m$-a.e.
\end{enumerate}
\end{proposition}
\begin{proof}
The proofs of both results are analogous, so that we only give the  the proof of (1).
 
{\bf Step 1}. Suppose that $f_2(\cdot,u_1)\in L^1(E;\rho\cdot m)$.
Observe that
\[
-Au_1=f_2(\cdot,u_1)+\mu_2+(f_1(\cdot,u_1)-f_2(\cdot,u_1))+(\mu_1-\mu_2).
\]
Therefore, $u_1$ is a subsolution to $-Av=f_2(\cdot,v)+\mu_2$.
By  Remark \ref{rem5.1}, $u_1\le u_2\,\,m$-a.e.

{\bf Step 2}. The general case. By Theorem \ref{th4.1}, there exists a maximal solution  $u_1^*$
to $-Av=f_1(\cdot,v)+\mu_1$. By the same theorem, there exist sequences $\{u_1^n\}$, $\{u_2^n\}$
such that $u_1^n$ is a maximal solution to $-Av=f_1^n(\cdot,v)+\mu_1$, $u_2^n$ is a maximal solution to $-Av=f_2^n(\cdot,v)+\mu_2$,
and $u_1^n\searrow u^*_1$, $u_2^n\searrow u_2\,\,m$-a.e. Here $f_i^n(x,y)=\max\{-n\phi(x),f_i(x,y)\},\, x\in E, y\in\BR$, $i=1,2$,
and $\phi$ is as in Theorem \ref{th4.1}. By  Theorem \ref{th3.1}(3), $h\le  u_1\le u^n_1\le R\mu_1^+\,\,m$-a.e. By the assumptions made on $h$, and by Int),
$f^n_2(\cdot,u_1^n)\in L^1(E;\rho\cdot m)$. By Step 1, $u_1^n\le u_2^n\,\,m$-a.e. Hence, $u_1\le u_2\,\,m$-a.e.
\end{proof}

\section{The class of good measures and  the  reduction operator}
\label{sec6}

In this section we shall investigate the class of good measures. 
Our goal is to   provide some properties of the set   $\GG(f)$ and the mapping $\mu\mapsto\mu^*$.
The main results of this section are Theorem \ref{prop5.3}, in which, by applying some basic properties of the mentioned objects,
we prove an existence result for \eqref{eq1.1}, and Theorems \ref{th23}, \ref{th24}  devoted to continuity of the operator $\mu\mapsto\mu^*$
and built up from  it metric projection onto $\GG(f)$. Proposition \ref{cor5.2}  also deserves attention. It is the first result 
concerning the structure of $\GG(f)$. In the next section we  considerably strengthen this result (Theorem \ref{th7.1}).   

In what follows we set for given $\mu\in\MM_\rho$,
\[
\GG_{\le\mu}(f)=\{\nu\in\GG(f): \nu\le\mu\}\quad \GG_{\ge\mu}(f)=\{\nu\in\GG(f): \nu\ge\mu\}.
\]
Let us note that by Theorems \ref{th4.1},\ref{th4.2}, $\mu^*,\mu_*$ are well defined iff $\GG_{\le\mu}(f)\neq\emptyset, \GG_{\ge\mu}(f)\neq\emptyset$, respectively,
and then
\[
\mu^*=\sup \GG_{\le\mu}(f),\quad \mu_*=\inf \GG_{\ge\mu}(f).
\]
Let $\check\MM_\rho:=\{\mu\in\MM_\rho: \GG_{\le \mu}(f)\neq\emptyset\}$.
We call the mapping 
\[
\check\MM_\rho\ni \mu\longmapsto \mu^*\in\GG(f)
\]
the {\em reduction operator}.

\begin{remark}
Notice that if there exist a positive $g\in L^1(E;\rho\cdot m)$ and $M\le 0$ such that $f(x,y)\le g(x),\, x\in E, y\le M$,
then $\check\MM_\rho=\MM_\rho$. Indeed,  observe that $\underline u:=-R\mu^-$ (resp. $\bar u:=0$) is a subsolution
(resp. supersolution) to 
\[
-Au= f(\cdot,u)-\mu^-.
\]
Therefore, by Theorem \ref{th3.1}, there exists a solution $w$ to the above equation.
Thus, $-\mu^-\in\GG(f)$. Clearly, $-\mu^-\le \mu$. 

\end{remark}

\subsection{Basic properties of good measures and the reduction operator:  application to the existence problem}
\label{sec6.1}


\begin{proposition}
\label{prop5.1}
We have the following.
\begin{enumerate}
\item[(1)] $\MM^0_\rho\subset \GG(f)$.
\item[(2)] Suppose  that $\nu \in\MM_\rho$, $\mu \in\check\MM_\rho$, and $\mu\le \nu$. Then $\mu^*\le\nu^*$.
\item[(3)] If  $\mu,\nu\in\GG(f)$, then  $\mu\vee \nu,\mu\wedge\nu\in \GG(f)$.
\item[(4)] If $\mu$ is a positive measure, 
then $\mu^*\ge 0$.
\end{enumerate}
\end{proposition}
\begin{proof}
Ad (1). It follows from the existence result proved in  \cite{KR:NoDEA2}. Ad (2). By the assumptions there exist $\mu^*, \nu^*$. 
By Theorem \ref{th4.1}, $\mu^*\le\mu$, and so $\mu^*\le \nu$. By Theorem \ref{th4.1} again, $\mu^*\le \nu^*$. Ad (3).  
Since $\mu$ is a good measure, there exists a solution $u$ to (\ref{eq1.1}). Observe that $u$ is also a subsolution to (\ref{eq1.1}) with $\mu$ replaced by $\mu\vee\nu$. Thus, there exists $(\mu\vee\nu)^*$. By Theorem \ref{th4.1},
$(\mu\vee \nu)^*\le \mu\vee\nu$. On the other hand, by (2)
\[
\mu=\mu^*\le (\mu\vee\nu)^*,\quad \nu=\nu^*\le (\mu\vee\nu)^*.
\]
So that $(\mu\vee\nu)^*=\mu\vee\nu$. Thus, $\mu\vee\nu\in\GG(f)$. Analogous reasoning for the minimum gives that $\mu\wedge\nu\in\GG(f)$. Ad (4). By (1), $0\in\GG(f)$. We assumed that $0\le \mu$. Therefore by (2), $0=0^*\le\mu^*$.
\end{proof}

\begin{proposition}
\label{prop5.12}
Let $\mu,\nu\in\MM_\rho$ and $\mu\wedge \nu\in\check\MM_\rho$. Then $(\mu\wedge \nu)^*=\mu^*\wedge\nu^*$.
\end{proposition}

\begin{proof}
First observe that $\mu\wedge \nu\in\check\MM_\rho$ implies that $\mu, \nu\in\check\MM_\rho$. By Theorem \ref{th4.1}, $\mu^*\le\mu, \nu^*\le\nu$. Hence $\mu^*\wedge  \nu^*\le \mu\wedge \nu$. By (2) and (3) of Proposition \ref{prop5.1},
  $\mu^*\wedge  \nu^*\le (\mu\wedge \nu)^*$. On the other hand, by Proposition \ref{prop5.1}(2), $\mu^*\ge (\mu\wedge \nu)^*$ and
  $\nu^*\ge (\mu\wedge \nu)^*$. Thus,  $\mu^*\wedge  \nu^*\ge (\mu\wedge  \nu)^*$.
\end{proof}


\begin{theorem}
\label{prop5.3}
Assume that there exists a subsolution and  a supersolution to (\ref{eq1.1}). Then there exists a solution to (\ref{eq1.1}).
\end{theorem}
\begin{proof}
Thanks to the  assumptions made, there exist a subsolution $\underline u$ and a supersolution $\overline u$ to \eqref{eq1.1}.
Therefore, according to the definition of these objects, there exist two positive measures $\nu_1,\nu_2\in\MM_\rho$
such that 
\[
-A\underline u=f(\cdot,\underline u)+\mu-\nu_1,\quad -A\overline u=f(\cdot,\overline u)+\mu+\nu_2.
\]
In particular, $\mu-\nu_1\in\GG_{\le\mu}(f)$, $\mu+\nu_2\in\GG_{\ge\mu}(f)$. Thus, there exist $\mu_*, \mu^*$. 
Obviousely, $\mu_*,\mu^*\in\GG(f)$.  Let $\overline w$ be a maximal solution to 
$-Av=f(\cdot,v)+\mu_*$ (see Remark \ref{rem5.1}). By Theorems \ref{th4.1}, \ref{th4.2}, $\mu^*\le\mu\le\mu_*$. Therefore, by Proposition \ref{prop.cop}, $u^*\le \overline w$. Observe that $u_*$ is a subsolution to (\ref{eq1.1}) and $\overline w$ is a supersolution to (\ref{eq1.1}).
By Theorem  \ref{th3.1}, there exists a solution to (\ref{eq1.1}).
\end{proof}

As a corollary to the above result, we obtain the following useful properties of the set $\GG(f)$.

\begin{proposition}
\label{prop5.5}
Assume that $\mu_1,\mu_2\in \GG(f)$, $\mu\in\MM_\rho$ and $\mu_1\le\mu\le\mu_2$.
Then $\mu\in \GG(f)$.
\end{proposition}
\begin{proof}
Since $\mu_1,\mu_2\in \GG(f)$. There exists a solution $u_1$ to (\ref{eq1.1}) with $\mu$ replaced by $\mu_1$,
and a solution $u_2$ to (\ref{eq1.1}) with $\mu$ replaced by $\mu_2$. Since $\mu_1\le\mu\le\mu_2$, $u_1$ is a subsolution to (\ref{eq1.1})
and $u_2$ is a supersolution to (\ref{eq1.1}).  By Theorem \ref{prop5.3}, there exists a solution to (\ref{eq1.1}). So, $\mu\in\GG(f)$.
\end{proof}



\begin{corollary}
\label{cor5.6}
Let $\mu\in\check\MM_\rho$. Then $\mu\in \GG(f)$ if and only if $\mu^+\in\GG(f)$.
\end{corollary}
\begin{proof}
If $\mu\in\GG(f)$, then by Proposition \ref{prop5.1}.(3), $\mu^+\in\GG(f)$. Suppose that  $\mu^+\in \GG(f)$. We have $\mu\le\mu^+$.
At the same time, since $\mu\in\check\MM_\rho$, there exists $\nu\in \GG(f)$ such that $\nu\le\mu$.
Therefore, by Proposition \ref{prop5.5}, $\mu\in \GG(f)$.
\end{proof}

\begin{proposition}
\label{prop5.6}
We have $\GG(f)+\MM_{\rho}^0\subset \GG(f)$.
\end{proposition}
\begin{proof}
Let $\gamma\in\GG(f)$ and $\beta\in\MM_{\rho}^0$.  Write
\[
\gamma=(\gamma+\beta)+(-\beta).
\]
Once we  show that for any $\mu\in\MM_\rho, \nu\in \MM^0_{\rho}$ such that $\mu+\nu\in \GG(f)$ we have that $\mu\in \GG(f)$, then
we conclude the result by taking $\mu=\gamma+\beta, \nu=-\beta$. So, let  $\mu\in\MM_\rho, \nu\in \MM^0_{\rho}$ and $\mu+\nu\in \GG(f)$. Set
\[
\overline \sigma =\max\{\mu+\nu,\mu_d\},\quad \underline\sigma =\min\{\mu+\nu,\mu_d\}.
\]
Since $\mu+\nu\in \GG(f)$, we get by Proposition \ref{prop5.1} (1),(3) that $\underline \sigma, \overline \sigma\in\GG(f)$.
Observe that
\[
\overline \sigma=\mu_d+\max\{\mu_c+\nu,0\}=\mu_d+(\mu_c+\nu)^+=\mu_d+\mu^+_c+\nu^+\ge \mu,
\]
and
\[
\underline \sigma=\mu_d+\min\{\mu_c+\nu,0\}=\mu_d-(\mu_c+\nu)^-=\mu_d-\mu^-_c-\nu^-\le \mu.
\]
Thus, $\underline\sigma\le \mu\le\overline\sigma$.  By Proposition \ref{prop5.5}, $\mu\in\GG(f)$.
\end{proof}

\begin{corollary}
\label{cor5.1}
We have that $\mu\in \GG(f)$ if and only if $\mu_c\in\GG(f)$.
\end{corollary}

\subsection{Further properties of the reduction operator and good measures - the class of admissible measures}
\label{sec6.2}

We let 
\[
\mathcal A(f)=\{\mu\in\MM_{\rho}: f(\cdot,R\mu)\in L^1(E;\rho\cdot m)\}.
\]
Elements of $\mathcal A(f)$ shall be called {\em admissible measures}.

\begin{proposition}
\label{cor5.2}
 We have  $\mathcal A(f)+L^1(E;\rho\cdot m)= \GG(f)$.
\end{proposition}
\begin{proof}
The inclusion $"\subset"$ follows directly from Proposition \ref{prop5.6}.
Suppose that $\mu\in\GG(f)$. Therefore, there exists a solution $u$ to \eqref{eq1.1}.
Set $\nu:=\mu+f(\cdot,u)$. Then $\nu\in\mathcal A(f)$ since $R\nu=u$ and by the very definition of a solution to
\eqref{eq1.1}, we have $f(\cdot,u)\in L^1(E,\rho\cdot m)$. Thus, $\mu=\nu-f(\cdot,u)\in \mathcal A(f)+L^1(E;\rho\cdot m)$.
\end{proof}

\begin{proposition}
\label{prop5.7}
The set $\GG(f)$ is closed in $(\MM_\rho, \|\cdot\|_{\MM_\rho})$.
\end{proposition}
\begin{proof}
Let $(\mu_n)_{n\ge 1}\subset \GG(f)$.  Let $\mu\in\MM_\rho$ and $\|\mu_n-\mu\|_{\MM_\rho}\rightarrow 0,\, n\rightarrow \infty$.
Set $\mu_0=0$. We may assume that $\sum_{n\ge 1}\|\mu_{n+1}-\mu_{n}\|_{\MM_\rho}<\infty$. Then $\mu=\sum_{n\ge0}(\mu_{n+1}-\mu_n)$,
where the limit is understood in the norm $\|\cdot\|_{\MM_\rho}$. Observe that
\[
\underline\sigma:=-\sum_{n\ge 0}(\mu_{n+1}-\mu_n)^-\le\mu_n\le\sum_{n\ge 0}(\mu_{n+1}-\mu_n)^+=:\overline \sigma
\]
Since $\mu_n\in\GG(f)$, there exist $\underline\sigma^*, \overline \sigma^*$ and by Proposition \ref{prop5.1}(2), $\underline\sigma^*\le\mu_n\le\overline \sigma^*$. Letting $n\rightarrow \infty$ yields $\underline\sigma^*\le\mu\le\overline \sigma^*$.
By Proposition \ref{prop5.5}, $\mu\in\GG(f)$.
\end{proof}

\begin{proposition}
\label{prop5.10}
Assume that  $\mu\in \check\MM_\rho$. Then $|\mu^*|\le |\mu|$.
\end{proposition}
\begin{proof}
By Theorem \ref{th3.1}(3), $u^*\le w$, where $w=R\mu^+$, and $v\le u^*$, where $v=-R\mu^-$.
By the inverse maximum principle $-\mu_c^-\le(\mu^*)_c\le \mu^+_c$. Hence $|(\mu^*)_c|\le|\mu_c|$. Since $(\mu^*)_d=\mu_d$,
we get the result. 
\end{proof}

\begin{corollary}
\label{cor5.3}
Assume that $\mu,\nu\in \check\MM_\rho$,   and  $\mu\bot\nu$. Then $\mu^*\bot \nu^*$.
\end{corollary}

\begin{proposition}
\label{prop5.13}
Let $\mu,\nu\in \check\MM_\rho$, and $\mu\bot\nu$. Then $(\mu+\nu)^*=\mu^*+\nu^*$.
\end{proposition}
\begin{proof}
First we show that $(\mu+\nu)^*$ is well defined and $\mu^*+\nu^*\le (\mu+\nu)^*$. Clearly, $\mu^*+\nu^*\le \mu+\nu$.
So, it is enough to prove that $\mu^*+\nu^*\in\GG(f)$ since then $\GG_{\mu+\nu}(f)\neq\emptyset$, and, by Proposition \ref{prop5.1}.(2), $\mu^*+\nu^*\le (\mu+\nu)^*$. By Proposition \ref{prop5.10}, $\mu^*\bot\nu^*$. Thus
\[
\mu^*\vee \nu^*=(\mu^*+\nu^*)^+,\quad \mu^*\wedge  \nu^*=-(\mu^*+\nu^*)^-\le\mu^*+\nu^*\le\mu+\nu.
\]
Therefore, by Proposition \ref{prop5.1}.(3), $\GG_{\mu+\nu}(f)\neq\emptyset$ and $(\mu^*+\nu^*)^+\in\GG(f)$. Hence,  by Corollary \ref{cor5.6}, $\mu^*+\nu^*\in \GG(f)$. 
For the  inequality $\mu^*+\nu^*\ge (\mu+\nu)^*$ observe that
\begin{equation}
\label{eq5.1}
(\mu+\nu)^*=s_\mu\cdot\mu+s_\nu\cdot\nu,
\end{equation}
where 
\[
s_\mu=\frac{d(\mu+\nu)^*}{d(|\mu|+|\nu|)}\cdot \frac{d|\mu|}{d\mu}, \quad s_\nu=\frac{d(\mu+\nu)^*}{d(|\mu|+|\nu|)}\cdot \frac{d|\nu|}{d\nu}.
\]
By Proposition \ref{prop5.10}, $s_\mu, s_\nu$ are well defined and $|s_\mu|\le 1, |s_\nu|\le1$.
Therefore, from  (\ref{eq5.1}) and the fact that $\mu\bot\nu$, we infer that
\[
-[(\mu+\nu)^*]^-\le  s_\mu\cdot\mu \le [(\mu+\nu)^*]^+,\quad   -[(\mu+\nu)^*]^-\le s_\nu\cdot\nu\le [(\mu+\nu)^*]^+.
 \]
By Proposition \ref{prop5.1}(3), $ [(\mu+\nu)^*]^+, -[(\mu+\nu)^*]^-\in\GG(f)$. Therefore, by Proposition \ref{prop5.5}, $s_\mu\cdot\mu,s_\nu\cdot\nu\in\GG(f)$.
By (\ref{eq5.1}), $s_\mu\cdot\mu+s_\nu\cdot\nu\le \mu+\nu$. From this and the fact that $\mu\bot\nu$, we conclude that  $s_\mu\cdot\mu\le \mu$, $s_\nu\cdot\nu \le\nu$. Consequently, since $s_\mu\cdot\mu,s_\nu\cdot\nu\in\GG(f)$, we have $s_\mu\cdot\mu\le \mu^*$, $s_\nu\cdot\nu \le\nu^*$.
This combined with (\ref{eq5.1}) gives $(\mu+\nu)^*\le\mu^*+\nu^*$. 
\end{proof}

\begin{corollary}
\label{cor5.21}
Let $\mu\in\check\MM_\rho$ and $A\in\BB(E)$.  Then $(\mu_{\lfloor A})^*=\mu^*_{\lfloor A}.$
\end{corollary}
\begin{proof}
First we show that $\GG_{\mu_{\lfloor A}}(f)\neq\emptyset$. For this it is enough to prove that $\nu\in\GG_{\le\mu}(f)$ implies $\nu_{\lfloor A}\in\GG_{\le \mu_{\lfloor A}}(f)$. But this follows easily  from Proposition \ref{prop5.1}.(3) and Proposition \ref{prop5.5} since $-\nu^-\le\nu_{\lfloor A}\le\nu^+$.
By Proposition \ref{prop5.13},
\[
(\mu^*)_{\lfloor A}+(\mu^*)_{\lfloor A^c}=(\mu_{\lfloor A})^*+(\mu_{\lfloor A^c})^*
\]
Applying Proposition \ref{prop5.10} yields $|(\mu^*)_{\lfloor A}|, |(\mu_{\lfloor A})^*| \le|\mu|_{\lfloor A} $ and 
$|(\mu^*)_{\lfloor A^c}|, |(\mu_{\lfloor A^c})^*| \le|\mu|_{\lfloor A^c} $. From this and the above equality, we get the result.
\end{proof}

\begin{corollary}
Let $\mu,\nu\in\check\MM_\rho$.  Then $(\mu\vee \nu)^*=\mu^*\vee\nu^*$.
\end{corollary}
\begin{proof}
It is enough to repeat step by step  the proof of  \cite[Theorem 4.9]{BMP}.
\end{proof}

\begin{corollary}
\label{cor5.6abc}
Let $\mu\in\check\MM_\rho$ and $\nu\in\MM^0_{\rho}$.  Then $(\mu+\nu)^*=\mu^*+\nu$.
\end{corollary}
\begin{proof}
Observe that by Proposition \ref{prop5.1}.(1), $\GG_{\mu+\nu}(f)\neq \emptyset$. Next, by Proposition \ref{prop5.13} and Proposition  \ref{prop5.1}.(1),
\[
(\mu+\nu)^*=(\mu_c)^*+(\mu_d+\nu)^*=(\mu_c)^*+\mu_d+\nu=(\mu_c)^*+\mu^*_d+\nu=(\mu_c+\mu_d)^*+\nu=\mu^*+\nu.
\]
\end{proof}

\begin{corollary}
Let $\mu\in\check\MM_\rho$.  Then $(\mu_c)^*=(\mu^*)_c$.
\end{corollary}
\begin{proof}
Let $\beta\in \GG_{\le\mu}(f)$. Then, by Corollary \ref{cor5.1}, $\beta_c\in\GG_{\le\mu_c}(f)$. Thus, $\GG_{\le\mu_c}(f)\neq \emptyset$.
By Corollary \ref{cor5.6abc} and Theorem \ref{th4.1},
\[
(\mu^*)_c=\mu^*-(\mu^*)_d=\mu^*-\mu_d=(\mu_c+\mu_d)^*-\mu_d=(\mu_c)^*.
\]
\end{proof}

From now on for every $\mu\in\MM_\rho$ without ambiguity we may write $\mu^*_c$.

\begin{proposition}
\label{prop5.17}
Let $\mu\in\check\MM_\rho$. Then
\[
\mu^*=\mu_d-\mu_c^-+(\mu^{+}_c)^*.
\]
\end{proposition}
\begin{proof}
Let $\beta\in\GG_{\le\mu}(f)$. Then $-\beta^-_c\le-\mu^-_c$. By Proposition \ref{prop5.1}.(3) and Corollary \ref{cor5.1}, $-\beta^-_c\in\GG(f)$.
Thus, $\GG_{\le-\mu^-_c}(f)\neq\emptyset$. By Proposition \ref{prop5.5}, $(-\mu^-_c)^*=-\mu^-_c$. From this and  Corollary \ref{cor5.6abc},
we get
\[
\mu^*=\mu_d+(\mu_c^+)^*+(-\mu^-_c)^*=\mu_d-\mu_c^-+(\mu^{+}_c)^*.
\]
\end{proof}

\subsection{Continuity  of the reduction operator}
\label{sec6.3}

The main result of the present subsection is  continuity of the reduction operator
with respect to the norm $\|\cdot\|_{\MM_\rho}$. Note that if we assume additionally  that $f$
is non-increasing with respect to the second variable, then the reduction operator is even Lipschitz continuous
(see \cite[Theorem 5.10]{K:CVPDE}).

\begin{lemma}
\label{lm5.1}
Let $\mu_n\in \MM_\rho,\, n\ge 1$, $\mu\in \check\MM_\rho$, and $\mu\le\mu_{n+1}\le\mu_n,\, n\ge 1$.
Assume that  $\mu_n\rightarrow \mu$ in the norm $\|\cdot\|_{\MM_\rho}$.
Then $\mu_n^*\rightarrow \mu^*$ in the norm $\|\cdot\|_{\MM_\rho}$. 
\end{lemma}
\begin{proof}
By Proposition \ref{prop5.1}, $\{\mu_n^*\}$ is a nondecreasing sequence and
\[
\mu^*\le \mu^*_n\le \mu_n,\quad n\ge 1.
\]
Since  $\{\mu_n^*\}$ is  nondecreasing, we may set $\beta=\lim_{n\rightarrow \infty}\mu^*_n$, where the limit is understood
in the norm $\|\cdot\|_{\MM_\rho}$. Letting $n\rightarrow \infty$ in the above inequality yields $\mu^*\le\beta\le\mu$.
Since $\mu^*\le\beta\le\mu^*_1$, we have by Proposition \ref{prop5.5} that $\beta\in\GG_{\le\mu}(f)$. Thus, $\beta=\mu^*$. 
\end{proof}

\begin{theorem}
\label{th23}
Let $\mu,\mu_n\in\MM_\rho,\, n\ge 1$. Assume that $\GG_{\le\mu}(f)\neq\emptyset$, $\GG_{\le\mu_n}(f)\neq\emptyset,\, n\ge 1$, and $\mu_n\rightarrow \mu$ in the norm $\|\cdot\|_{\MM_\rho}$. Then, $\mu^*_n\rightarrow \mu^*$ in the norm $\|\cdot\|_{\MM_\rho}$.
\end{theorem}
\begin{proof}
All the convergences of measures considered in the proof below  will be  understood in the norm $\|\cdot\|_{\MM_\rho}$.

{\bf Step 1.} We assume additionally that $0\le\mu\le\mu_n,\, n\ge 1$. Let $(n_k)$ be a subsequence of $(n)$. By \cite[Proposition 4.2.4]{Peressini}, there exists a further subsequence $(n_{k_l})$, and positive $\beta\in\MM_\rho$ such that 
\[
|\mu_{n_{k_l}}-\mu|\le\frac{1}{l}\beta,\quad l\ge 1.
\]
Thus,
\[
\mu\le \mu_{n_{k_l}}\le\mu+\frac1l\beta,\quad l\ge 1.
\]
By Proposition \ref{prop5.1},
\[
\mu^*\le \mu^*_{n_{k_l}}\le(\mu+\frac1l\beta)^*,\quad l\ge 1.
\]
By Lemma \ref{lm5.1}, $(\mu+\frac1l\beta)^*\rightarrow \mu^*$. From this and the above inequality $\mu^*_{n_{k_l}}\rightarrow \mu^*$. Since $(n_k)$
was an arbitrary subsequence of $(n)$, we conclude that $\mu^*_n\rightarrow \mu^*$.

{\bf Step 2.} We assume additionally that $0\le\mu_n\le\mu,\, n\ge 1$. Then $\mu_n=s_n\cdot\mu$
for some Borel function $s_n$ on $E$ such that $0\le s_n\le 1$. Since $\mu_n\rightarrow \mu$, we have that $s_n\rightarrow 1$ in $L^1(E;\mu)$. Let $a\in (0,1)$. Observe that for any positive $\nu\in\MM_\rho$,
\begin{equation}
\label{eq5.31}
a\nu^*\le (a\nu)^*.
\end{equation} 
Set $A_n=\{s_n\ge a\}$. Then
\begin{equation}
\label{eq5.32}
a\mu_{\lfloor A_n}\le \mu_n\le\mu,\quad n\ge 1.
\end{equation} 
Moreover,
\begin{equation}
\label{eq5.33}
\mu(A_n^c)=\mu(|1-s_n|>1-a)\le \frac{1}{1-a}\int_E|1-s_n|\,d\mu\rightarrow 0.
\end{equation}
By (\ref{eq5.31}), (\ref{eq5.32}), Proposition \ref{prop5.1} and Corollary \ref{cor5.21}
\[
a(\mu^*)_{\lfloor A_n}\le \mu^*_n\le\mu^*,\quad n\ge 1.
\]
From this and (\ref{eq5.33}), we get
\[
a\mu^*\le\liminf_{n\rightarrow \infty}\mu^*_n\le\limsup_{n\rightarrow \infty}\mu^*_n\le\mu^*.
\]
Since $a\in (0,1)$ was arbitrary, we get $\mu^*_n\rightarrow \mu^*$.

{\bf Step 3.} We assume additionally that $\mu_n\ge 0,\, n\ge 1$. Then observe that
\begin{equation}
\label{eq5.34}
0\le\mu\wedge\mu_n\le\mu_n\le\mu\vee\mu_n,\quad n\ge 1.
\end{equation}
Clearly, $\mu\wedge\mu_n\rightarrow \mu$ and $\mu\vee\mu_n\rightarrow \mu$. Therefore, by Step 1 and  Step 2, $(\mu\wedge\mu_n)^*\rightarrow \mu^*$ and $(\mu\vee\mu_n)^*\rightarrow \mu^*$.  By (\ref{eq5.34}) and Proposition \ref{prop5.1},
\[
(\mu\wedge\mu_n)^*\le\mu^*_n\le(\mu\vee\mu_n)^*,\quad n\ge 1.
\]
Thus, $\mu_n^*\rightarrow  \mu^*$.

{\bf Step 4.} The general case.  By Proposition \ref{prop5.17},
\begin{equation}
\label{eq5.35}
\mu^*_n=(\mu_n)_d-(\mu^-_n)_c+(\mu^+_n)^*_c,\quad  \mu^*=\mu_d-\mu^-_c+(\mu^+_c)^*.
\end{equation}
Since $\mu_n\rightarrow \mu$, we have $(\mu_n)_d\rightarrow \mu_d$, $(\mu^-_n)_c\rightarrow \mu^-_c$, $(\mu^+_n)_c\rightarrow \mu^+_c$. By Step 3, 
$(\mu^+_n)^*_c\rightarrow (\mu^+_c)^*$.  As a result, by (\ref{eq5.35}), $\mu^*_n\rightarrow \mu^*$.
\end{proof}
\subsection{Existence and continuity of the metric projection onto good measures}
In what follows, in order to emphasize the dependence of the reduction operator on the nonlinearity $f$,
we shall denote by $\mu^{*,f}$ and $\mu_{*,f}$ the measure $\mu^*$ and $\mu_*$ appearing in the assertions of Theorem \ref{th4.1} and Theorem \ref{th4.2}, respectively.

We let
\[
\Pi_{f}(\mu):=\mu^{*,f},\quad \mu\in\check\MM_\rho.
\]
We also let $\hat\MM_\rho:=\{\mu\in\MM_\rho: \GG_{\ge\mu}\neq\emptyset\}$.
By Proposition \ref{prop5.5}, $\check\MM_\rho\cap \hat\MM_\rho=\GG(f)$. Thus, we may extend operator $\Pi_{f}$:
\begin{equation}
\Pi_{f}(\mu)=
\begin{cases}
\mu^{*,f},\quad\mu\in\check\MM_\rho\\
\mu_{*,f},\quad \mu\in\hat\MM_\rho.
\end{cases}
\end{equation}
Since for $\mu\in  \check\MM_\rho\cap \hat\MM_\rho$, we have $\mu^{*,f}=\mu_{*,f}=\mu$, the operator $\Pi_{f}$
is well defined on $ \check\MM_\rho\cup \hat\MM_\rho$.

We denote
\[
\tilde f(x,y):= -f(x,-y),\quad x\in E, y\in\BR.
\]
We get at once that if $f$ satisfies one of the conditions Int), Car), qM), Sig), then $\tilde f$ satisfies it too.

\begin{remark}
\label{rem4.1}
Observe that by Theorems \ref{th4.1}, \ref{th4.2}
for any $\mu\in\hat\MM_\rho$,
\[
-\Pi_{\tilde f}(-\mu)=-(-\mu)^{*,\tilde f}=\mu_{*,f}.
\]
\end{remark}

\begin{proposition}
\label{prop25}
The mapping
\[
\Pi_{f}:\check\MM_\rho\cup \hat\MM_\rho\to \GG(f)
\]
is the metric projection onto $\GG(f)$. Moreover,
\begin{equation}
\label{eq4.phat}
|\mu-\Pi_{f}(\mu)|\le (\mu-\nu)^+,\quad\mu\in\check\MM_\rho,\,\nu \in \GG(f), 
\end{equation}
and 
\begin{equation}
\label{eq4.pcheck}
|\mu-\Pi_{f}(\mu)|\le (\nu-\mu)^+,\quad\mu\in\hat\MM_\rho,\,\nu \in \GG(f). 
\end{equation}
Furthermore,  for any $\mu\in \check\MM_\rho\cup \hat\MM_\rho$, the measure $\Pi_{f}(\mu)$
is the only one element of $\GG(f)$ satisfying 
\[
\|\mu-\Pi_{f}(\mu)\|_{\MM_\rho}=\inf_{\nu\in\GG(f)}\|\mu-\nu\|_{\MM_\rho}.
\]
\end{proposition}
\begin{proof}
Let $\nu\in\GG(f)$ and $\mu\in \check\MM_\rho$. The last relation implies that there exists $\beta\in\GG(f)$ such that $\beta\le \mu$.
Thus, $\beta\wedge\nu\le \mu\wedge\nu\le\nu$. By Proposition \ref{prop5.1} and Proposition \ref{prop5.5}, $\mu\wedge\nu\in \GG(f)$.
This in turn implies that $\mu\wedge \nu\le (\mu\wedge \nu)^{*,f}=\mu^{*,f}\wedge \nu$ (see Proposition \ref{prop5.12}).  Consequently,
\begin{equation}
\label{eq4.phat1}
|\mu-\Pi_{f}(\mu)|=|\mu-\mu^{*,f}|=\mu-\mu^{*,f}\le \mu-\mu^{*,f}\wedge\nu\le \mu-\mu\wedge\nu=(\mu-\nu)^+\le |\mu-\nu|.
\end{equation}
Therefore, $\|\mu-\Pi_{f}(\mu)\|_{\MM_\rho}\le \|\mu-\nu\|_{\MM_\rho}$ for any $\nu\in\GG(f)$. This completes the proof of case $\mu\in\check\MM_\rho$.
Now, let $\mu\in \hat\MM_\rho$.  Observe that $-\mu\in \check\MM_\rho$. Therefore, by Remark \ref{rem4.1}
\[
\|\mu-\Pi_{f}(\mu)\|_{\MM_\rho}=\|\mu-(-\Pi_{\tilde f}(-\mu))\|_{\MM_\rho}=\inf_{\nu\in\GG(\tilde f)}\|\nu-(-\mu)\|_{\MM_\rho}=\inf_{\nu\in\GG(f)}\|-\nu+\mu\|_{\MM_\rho}.
\]
Applying \eqref{eq4.phat1} with $-\mu$ in place of $\mu$ yields \eqref{eq4.pcheck}. For the proof of the last assertion 
of the proposition
suppose that $\mu\in \check\MM_\rho$ (the proof of the second case is analogous) and suppose that $\nu\in\GG(f)$
realizes the distance between  $\mu$ and  $\GG(f)$. Notice that
\begin{equation}
\label{eqj.j1}
\|\mu-\mu\wedge\nu\|_{\MM_\rho}=\|(\mu-\nu)^+\|_{\MM_\rho}.
\end{equation}
Since $\mu\in\check\MM_\rho$, there exists $\gamma\in\GG(f)$ such that $\gamma\le\mu$. 
Hence, $\gamma\wedge\nu\le \mu\wedge\nu\le \nu$. By Proposition \ref{prop5.5}, $\mu\wedge\nu\in\GG(f)$.
Therefore, since $\nu$ realizes the distance between $\mu$ and $\GG(f)$, we have
\begin{equation}
\label{eqj.j2}
\|\mu-\nu\|_{\MM_\rho}\le \|\mu-\mu\wedge\nu\|_{\MM_\rho}.
\end{equation}
This combined with \eqref{eqj.j1} yields $(\mu-\nu)^-=0$, so that $\nu\le\mu$. 
The last inequality is crucial since it implies that $\nu\le \mu^{*,f}$. We also have, $\Pi_{f}(\mu)=\mu^{*,f}\le \mu$. On the other hand, since  $\Pi_{f}(\mu)$
and $\nu$ realize the distance between $\mu$ and $\GG(f)$, we have
\[
\|\mu-\Pi_{f}(\mu)\|_{\MM_\rho}=\|\mu-\nu\|_{\MM_\rho}.
\]
Therefore, we deduce at once that $\Pi_{f}(\mu)=\nu$.
\end{proof}
Finally, we define the operator
\[
\Pi_{f}:\MM_\rho\to\GG(f)
\]
as follows: 
\begin{align*}
\Pi_{f}(\mu)&:=\Pi_{f}(\mu^+)+\Pi_{f}(-\mu^-)\\&
=\Pi_{f}(\mu^+)-\Pi_{\tilde f}(\mu^-)=(\mu^+)^{*,f}-(\mu^-)^{*,\tilde f}=(\mu^+)^{*,f}+(-\mu^-)_{*,f}.
\end{align*}
By Propositions \ref{prop5.10}, \ref{prop5.13}, $\Pi_{f}(\mu)\in \GG(f)$ for any $\mu\in \MM_\rho$.

\begin{theorem}
\label{th24}
$\Pi_{f}:\MM_\rho\to \GG(f)$ is  a continuous metric projection onto $\GG(f)$.
Moreover, 
\[
|\mu-\Pi_{f}(\mu)|\le |\mu-\nu|,\quad \mu\in\MM_\rho,\, \nu\in\GG(f).
\]
Furthermore, if $Q:\MM_\rho\to \GG(f)$ is a metric projection, with the property: $\mu\bot\nu$ implies $Q(\mu+\nu)=Q(\mu)+Q(\nu)$,
then $Q=\Pi_{f}$.
\end{theorem}
\begin{proof}
Continuity of $\Pi_{f}$ follows from Theorem  \ref{th23}. Let $\nu\in\GG(f)$ and $\mu\in \MM_\rho$.
Clearly, $\Pi_{f}(\mu)\in \GG(f)$. Moreover, by Proposition \ref{prop25}, for any $\nu_1\in \GG(f)$, $\nu_2\in\GG(\tilde f)= -\GG(f)$
\[
|\mu-\Pi_{f}(\mu)|=|\mu^+-\Pi_{f}(\mu^+)|+|\mu^--\Pi_{\tilde f}(\mu^-)|\le (\mu^+-\nu_1)^++(\mu^--\nu_2)^+.
\]
Let $\nu\in\GG(f)$. Then $\nu^+\in\GG(f)$ and $\nu^-\in-\GG(f)$. Therefore,
\[
|\mu-\Pi_{f}(\mu)|\le (\mu^+-\nu^+)^++(\mu^--\nu^-)^+\le |\mu-\nu|.
\]
This implies the inequality asserted in the theorem, and the fact that $\Pi_{f}$ is the metric projection onto $\GG(f)$.
For the last assertion of the theorem, observe that operator $\Pi_{f}$ shares additivity  property
formulated in the assertion of the theorem for $Q$. Therefore, if $\Pi_{f}=Q$ on $\check\MM_\rho\cup\hat\MM_\rho$,
then $\Pi_{f}=Q$ on $\MM_\rho$. The fact that  $\Pi_{f}=Q$ on $\check\MM_\rho\cup\hat\MM_\rho$ follows easily from 
Proposition \ref{prop25}.
\end{proof}

\section{Characterization of the class of good measures}
\label{sec7}

\begin{proposition}
\label{propj.j1}
Assume that $\{\mu_n\}$ is a sequence of positive Borel measures such that 
$\sup_{n\ge 1}R\mu_n<\infty\,\,$ q.e.
Suppose that $R\mu_n\to 0\,\,m$-a.e. Then there exists a subsequence (not relabeled)
such that $R\mu_n\to 0$ q.e.
\end{proposition}
\begin{proof}
By \cite[Lemma 5.1]{K:JFA1}, there exists a subsequence (not relabeled) such that $k\wedge R\mu_n\to 0$ q.e. for any $k\ge 1$.
As a result, since $\sup_{n\ge 1}R\mu_n<\infty\,\,$ q.e., we infer from this convergence that up to subsequence $R\mu_n\to 0\,\,$q.e.
\end{proof}

\begin{theorem}
\label{th7.1}
Let  $\mu\in \MM_\rho$. 
\begin{enumerate}
\item[(1)]
$\mu\in\GG(f)$ if and only if
there exists a sequence $\{g_n\}\subset L^1(E;\rho\cdot m)$ such that
\begin{enumerate}
\item[(i)] $g_n+\mu\in\mathcal A(f),\, n\ge 1$.
\item[(ii)] $g_n\rightarrow 0$ in $L^1_{loc}(E;\rho\cdot m)$.
\end{enumerate}
\item[(2)]
Assume that $\rho$ is bounded and there exists $\varepsilon>0$ such that $\sup_{|y|\le\varepsilon}|f(\cdot,y)|\in L^1(E;\rho\cdot m)$, 
then $\mu\in\GG(f)$ if and only if
there exists a sequence $\{g_n\}\subset L^1(E;\rho\cdot m)$ such that condition (i) and the following one
\begin{enumerate}
\item[(ii')] $g_n\rightarrow 0$ in $L^1(E;\rho\cdot m)$
\end{enumerate}
hold.
\end{enumerate}
Furthermore, in both cases ((1) and (2)), if $\mu$ is positive (resp. negative) then $g_n$ may be taken negative (resp. positive).
\end{theorem}
\begin{proof}
Sufficiency (in both cases) follows from Corollary \ref{cor5.21} and Proposition \ref{prop5.7}.
Let $\mu\in\GG(f)$. Set
\[
f_{n,m}(x,y):=\frac1mf^+(x,y)-\frac1nf^-(x,y),\quad x\in E, y\in\BR.
\]
Since $\mu\in \GG(f)$, there exists a solution $u$ to $-Av=f(\cdot,v)+\mu$. Hence
\[
-Au=f_{n,m}(\cdot,u)+(\mu-f_{n,m}(\cdot,u)+f(\cdot,u)).
\]
Clearly, $f_{n,m}(\cdot,u)\in L^1(E;\rho\cdot m)$. Therefore, from the above equation, $\mu-f_{n,m}(\cdot,u)+f(\cdot,u)\in\GG(f_{n,m})$.
Thus, by Proposition \ref{prop5.6}, $\mu\in \GG(f_{n,m})$. Consequently, by Theorem \ref{th4.1}, there exists a maximal solution $u_{n,m}$ to
\[
-Av=f_{n,m}(\cdot,v)+\mu.
\]
By Proposition \ref{prop.cop}, $u_{n,m}\le u_{n+1,m}$, $u_{n,m}\ge u_{n,m+1},\, n,m\ge 1$ q.e.
Set 
\[
w_n:=\lim_{m\rightarrow \infty}u_{n,m}=\inf_{m\ge 1} u_{n,m},\quad z_m:=\lim_{n\rightarrow \infty}u_{n,m}=\sup_{n\ge 1} u_{n,m}\quad \text{q.e.}
\]
and
\[
w:=\lim_{n\rightarrow \infty}w_n=\sup_{n\ge 1} w_n,\quad z:=\lim_{m\rightarrow \infty}z_m=\inf_{m\ge 1} z_m\quad \text{q.e.}
\]
Observe that 
\[
0\le u_{n,m}^+\le u_{n,1}^+,\quad m\ge 1\quad \text{q.e.}
\]
Thus, by Int),
\[
f^-(\cdot,u_{n,m})=|f(\cdot,u_{n,m}^+)|\le \sup_{0\le y\le u^+_{n,1}}|f(\cdot,y)|\in L^1(E;\rho\cdot m).
\]
From this, we conclude that, up to subsequence,
\[
Rf^-(\cdot,u_{n,m})\rightarrow Rf^-(\cdot,w_n)\quad m\mbox{-a.e.}
\]
By Proposition \ref{prop3.3},
\begin{equation}
\label{eq.mt1}
|u_{n,m}|+R|f_{n,m}(\cdot,u_{n,m})|\le R|\mu|\quad \mbox{q.e.}
\end{equation}
By \cite[Lemma 94, page 306]{DellacherieMeyer}, up to subsequence, $\frac 1m Rf^+(\cdot,u_{n,m})\rightarrow e_n,\, m\rightarrow \infty,\, m$-a.e. for some excessive
function $e_n$. By  \cite[Proposition 3.9]{GetoorGlover}, there exists a positive Borel measure $\beta_n$ such that $e_n=R\beta_n$. Therefore,
\begin{equation}
\label{eq.mt2}
w_n=R\beta_n-\frac1n Rf^-(\cdot,w_n)+R\mu\quad m\mbox{-a.e.}
\end{equation}
Now, we shall show that $w_n$ is quasi-continuous and $\beta_n\bot Cap$. 
Set $h:=R|\mu|$. By Remark \ref{rem3.1}(a), $h$ is quasi-continuous. Set
\[
\tau_k:=\inf\{t\ge 0:  h(X_t)\ge k\}\wedge k,\quad \hat\sigma_{k,j}:=\inf\{t\ge 0: \int_0^t \sup_{|y|\le k}|f(X_r,y)|\,dr\ge j\}.
\]
Since $h$ is quasi-continuous,
$\lim_{k\rightarrow \infty}\tau_k=\zeta$. By qM) and Corollary \ref{cor2.1}, $\lim_{j\rightarrow \infty}\hat \sigma_{k,j}=\zeta$
for any $k\ge 1$.
Thus, $\lim_{k\to \infty}\lim_{j\to \infty}\tau_{k,j}\nearrow \zeta$, where $\tau_{k,j}:=\tau_k\wedge \hat\sigma_{k,j}$.
By Lemma \ref{lm.m}, $\{\sigma_k\}$ is a reducing sequence for $h$, hence $\{\tau_{k,j}\}$ is a reducing sequence for $h$
for fixed $j\ge 1$.
Therefore, by (\ref{eq.mt1}), the last sentence also is in force  with $h$ replace by any of the following functions: $u_{n,m}, w_n, z_n, w, z$. 
By Lemma \ref{lm.m},
\begin{align}
\label{eq.mt3jj1}
\nonumber
u_{n,m}(X_{t\wedge\tau_{k,j}})=\mathbb E_x\Big[u_{n,m}(X_{\tau_{k,j}})&+\frac1m \int_{t\wedge\tau_{k,j}}^{\tau_{k,j}}f^+(\cdot,u_{n,m})(X_r)\,dr\\&-\frac1n \int_{t\wedge\tau_{k,j}}^{\tau_{k,j}}f^-(\cdot,u_{n,m})(X_r)\,dr
+\int_{t\wedge\tau_{k,j}}^{\tau_{k,j}}\,dA^{\mu_d}_r\big|\FF_{t\wedge\tau_{k,j}}\Big]\quad \mbox{q.a.s.}
\end{align}
By \cite[Lemma 6.1]{BDHPS}, for any $q\in (0,1)$ there exists $c_q>0$ such that
\begin{align}
\label{eq.mt3jj14}
\nonumber
\Big(\mathbb E_x&\sup_{t\le \tau_{k,j}}|u_{n,m}(X_t)-u_{n,l}(X_t)|^q\Big)^{1/q}\le c_q \mathbb E_x\Big[|u_{n,m}(X_{\tau_{k,j}})-u_{n,l}(X_{\tau_{k,j}})|\\&+\frac1m \int_0^{\tau_{k,j}}f^+(\cdot,u_{n,m})(X_r)\,dr
+\frac1l \int_0^{\tau_{k,j}}f^+(\cdot,u_{n,l})(X_r)\,dr\\&+\frac1n \int_0^{\tau_{k,j}}|f^-(\cdot,u_{n,m})(X_r)-f^-(\cdot,u_{n,l})(X_r)|\,dr\Big]
\quad \mbox{q.e.}
\end{align}
By Remark \ref{rem3.1}(c) and the choice of $\{\tau_{k,j}\}$, we obtain that the right-hand side of the above inequality tends to
zero as $m,l\to \infty$. Consequently, by Lemma \ref{lmj.jlm}, $w_n$ is quasi-continuous.
Taking $t=0$ in \eqref{eq.mt3jj1}, we get
\begin{align}
\label{eq.mt3}
\nonumber
u_{n,m}(x)=\mathbb E_xu_{n,m}(X_{\tau_{k,j}})&+\frac1m \mathbb E_x\int_0^{\tau_{k,j}}f^+(\cdot,u_{n,m})(X_r)\,dr\\&-\frac1n \mathbb E_x\int_0^{\tau_{k,j}}f^-(\cdot,u_{n,m})(X_r)\,dr
+\mathbb E_x\int_0^{\tau_{k,j}}\,dA^{\mu_d}_r\quad \mbox{q.e.}
\end{align}
Letting $m\rightarrow \infty$ and using the definition of $\tau_{k,j}$ and Remark \ref{rem3.1}(c) yields
\[
w_{n}(x)=\mathbb E_xw_{n}(X_{\tau_{k,j}})-\frac1n \mathbb E_x\int_0^{\tau_{k,j}}f^-(\cdot,w_{n})(X_r)\,dr+\mathbb E_x\int_0^{\tau_{k,j}}\,dA^{\mu_d}_r\quad\mbox{q.e.}
\]
Now, letting $j\to\infty$, and using quasi-continuity of $w_n$ and the fact that $(\tau_k)$
is a reducing sequence for $w_n$ (and $\tau_{k,j}\le\tau_k$), we find that 
\[
w_{n}(x)=\mathbb E_xw_{n}(X_{\tau_{k}})-\frac1n \mathbb E_x\int_0^{\tau_{k}}f^-(\cdot,w_{n})(X_r)\,dr+\mathbb E_x\int_0^{\tau_{k}}\,dA^{\mu_d}_r\quad\mbox{q.e.}
\]
On the other hand, since $w_n$ is quasi-continuous, we have by Remark \ref{rem3.1}(a)--(b) that, in fact, \eqref{eq.mt2}
holds q.e. Therefore, by \cite[Theorem 3.7]{K:CVPDE}, letting $k\rightarrow \infty$ in the above equation gives
\[
w_{n}=R(\mu+\beta_n)_c-\frac1n Rf^-(\cdot,w_n)+R\mu_d\quad\mbox{q.e.}
\]
Thus,
\[
w_{n}=R(\beta_n)_c-\frac1n Rf^-(\cdot,w_n)+R\mu\quad\mbox{q.e.}
\]
From this and (\ref{eq.mt2}), we conclude that $(\beta_n)_c=\beta_n$. Since $u_{n,m}\le u_{n,1},\, n,m\ge 1$ q.e., we have  $w_n\le u_{n,1},\,n\ge 1$ q.e.
Therefore, by the inverse maximum principle (see \cite[Theorem 6.1]{K:CVPDE}), $\beta_n+\mu_c\le\mu_c$. Hence, $\beta_n=0$. Consequently,
\begin{equation}
\label{eqj.j2}
w_{n}=-\frac1n Rf^-(\cdot,w_n)+R\mu\quad\mbox{q.e.}
\end{equation}
Repeating the reasoning (\ref{eq.mt1})--(\ref{eqj.j2}), with $u_{n,m}$ replaced by $w_n$ (and this time  letting $n\to\infty$)
and with $w_n$ replaced by $w$, we  find that  $w=R\mu$ q.e. Analogous  reasoning shows that $z=R\mu$ q.e.
Set $u_n=u_{n,n}, f_n=f_{n,n}$, then
\[
-Au_n=f_n(\cdot,u_n)+\mu.
\]
Observe that $w_n\le u_n\le z_n$ q.e. Thus, $u_n\rightarrow R\mu$ q.e. By (\ref{eq.mt1}),
\cite[Lemma 94, page 306]{DellacherieMeyer} and Proposition \ref{propj.j1}, there exist excessive functions $e_1,e_2$ such that, up to subsequence, 
\[
e^n_1:=\frac1n Rf^+(\cdot,u_n)\rightarrow e_1,\quad e_2^n:=\frac1n Rf^-(\cdot,u_n)\rightarrow e_2,\quad \mbox{q.e.}
\]
By  \cite[Proposition 3.9]{GetoorGlover} and once again (\ref{eq.mt1}),  there exist positive Borel measures $\beta_1,\beta_2$
such that $e_1=R\beta_1$, $e_2=R\beta_2$.
At the same time, since $u_n\rightarrow u$ q.e., we have $\frac1n Rf(\cdot,u_n)\rightarrow 0$ q.e.
Thus, $e_1=e_2$, and so $\beta_1=\beta_2$. By Lemma \ref{lm.m}
\[
e_1^n(x)=\mathbb E_xe_1^n(X_{\tau_k})+\frac1n\mathbb E_x\int_0^{\tau_k}f^+(\cdot,u_n)\quad\text{q.e.}
\]
By the choice of $\{\tau_k\}$ and Remark \ref{rem3.1}(c), we obtain, by letting $n\to \infty$, that $e_1(x)=\mathbb E_x e_1(X_{\tau_k})$ q.e.
Thus, by  \cite[Theorem 3.7]{K:CVPDE}, $e_1=R(\beta_1)_c$. Consequently, $R(\beta_1)_c=R\beta_1$, so that $(\beta_1)_c=\beta_1$.
By \cite[Proposition 3.7]{K:NoDEA}, there exists positive smooth measures $\lambda_n, \lambda\in \MM_\rho$ such that
\[
u^+_n=-e_2^n+R\mu^+-R\lambda_n,\quad u^+=R\mu^+-R\lambda.
\]
Letting $n\to \infty$ and using \eqref{eq.mt1} and  \cite[Proposition 3.9]{GetoorGlover}, we deduce that
there exists a positive measure $\lambda^0\in\MM_\rho$ such that
\[
u^+=-R\beta_1+R\mu^+-R\lambda^0,\quad u^+=R\mu^+-R\lambda.
\]
Thus, $\lambda^0+\beta_1=\lambda$. Since $\lambda$ is smooth and $(\beta_1)_c=\beta_1$, we conclude that $\beta_1=0$.
Consequently, $e_1=e_2=0$. As a result, we obtain that 
\[
R|f_n(\cdot,u_n)|\to 0\quad \text{q.e.}
\]
From this and \eqref{eq.mt1} we  infer that for any positive smooth measure $\nu$ such that $R\nu\le \rho$ we have
\begin{equation}
\label{eqj.j5}
\int_E |f_n(\cdot,u_n)|R\nu\,dm\to 0\quad\text{as}\quad n\to \infty.
\end{equation}
Set $F:=\{R|\mu|>\varepsilon\}$ and $h:=\rho\cdot\mathbf1_F$.
Let $e_h$ be the smallest excessive function less than or equal to $h$.
Clearly, $e_h\le |\rho|_\infty R|\mu|\wedge \rho$. Therefore, by \cite[Proposition 3.9]{GetoorGlover}, there exists a positive
measure $\nu\in\MM_\rho$ such that $e_h=R\nu$. Since $e_h$ is bounded, $\nu$ is a smooth measure.
Consequently, \eqref{eqj.j5} holds. We have
\begin{align*}
|f_n(\cdot,u_n)|\rho&=\mathbf1_F|f_n(\cdot,u_n)|\rho+\mathbf1_{F^c}|f_n(\cdot,u_n)|\rho
\\&\le
|f_n(\cdot,u_n)|e_h+\sup_{|y|\le\varepsilon}|f_n(\cdot,y)|\rho=|f_n(\cdot,u_n)|R\nu+\frac1n \sup_{|y|\le\varepsilon}|f(\cdot,y)|\rho.
\end{align*}
By \eqref{eqj.j5} and the assumptions made on $f$, we get the result. The last assertion of the theorem is obvious from the construction.
\end{proof}

We let $B_{L^1}(0,r):=\{u\in L^1(E;\rho\cdot m): \|u\|_{L^1(E;\rho\cdot m)}\le r\}$.

\begin{corollary}
\label{cor.main}
Under the notation and assumptions of Theorem \ref{th7.1}(2), we have
\begin{enumerate}
\item[(i)] for any $r>0$, $\mathcal A(f)+B_{L^1}(0,r)=\GG(f)$,
\item[(ii)] $cl \mathcal A(f)=\GG(f)$, where $cl$ denotes closure in the total variation norm $\|\cdot\|_\rho$.
\end{enumerate}
\end{corollary}
\begin{proof}
It follows directly from Theorem \ref{th7.1}.
\end{proof}

\begin{remark}
\label{cor.asym1}
Assume  that $g$ is a  function satisfying {\em Car),Sig),Int)}. Furthermore, assume that $f,g$ satisfy {\rm M)}.
Suppose  that there exist   constants $c_1,c_2,r>0$ such that 
\begin{equation}
\label{eq.asym1}
c_1\le\frac{|f(x,y)|}{|g(x,y)|}\le c_2,\quad m\mbox{-a.e.},\quad |y|\ge r.
\end{equation}
Then $\GG(f)=\GG(g)$.
\end{remark}
\begin{proof}
By  \eqref{eq.asym1}, we easily get that $\mathcal A(f)=\mathcal A(g)$. Therefore, by Corollary \ref{cor.main}, $\GG(f)=\GG(g)$.
\end{proof}

\begin{corollary}
\label{cor.asym12}
Under assumptions of Remark \ref{cor.asym1}, we have $\Pi_f=\Pi_g$. In particular, for positive $\mu\in\MM_\rho$, $\mu^{*,f}=\mu^{*,g}$.
\end{corollary}

\subsection*{Acknowledgements}
{\small This work was supported by Polish National Science Centre
(Grant No. 2017/25/B/ST1/00878).}

\end{document}